\documentclass[12pt,reqno]{amsart}

\usepackage{geometry}
\usepackage{hhline}

\usepackage{mathtools}

\usepackage{mathdots}

\usepackage{color}

\usepackage{pdfsync}

\usepackage{enumitem}

\usepackage{wasysym}

\usepackage{amssymb}

\usepackage{mathrsfs}

\usepackage{bbm}

\DeclareMathAlphabet{\mathpzc}{OT1}{pzc}{m}{it}

\usepackage[all]{xy}

\usepackage{tikz}
\usetikzlibrary{arrows,matrix,decorations.pathmorphing,decorations.pathreplacing,positioning,shapes.geometric,shapes.misc,decorations.markings,decorations.fractals,calc,patterns}

\usepackage{tikz-cd}

\usepackage{graphicx}

\usepackage{float}

\usepackage[bottom]{footmisc}

\usepackage{moreenum}

\usepackage{makecell}

\entrymodifiers={+!!<0pt,\fontdimen22\textfont2>}

\usepackage{scalerel,stackengine}
\stackMath
\newcommand\newcheck[1]{%
\savestack{\tmpbox}{\stretchto{%
  \scaleto{%
    \scalerel*[\widthof{\ensuremath{#1}}]{\kern-.6pt\bigwedge\kern-.6pt}%
    {\rule[-\textheight/2]{1ex}{\textheight}}
  }{\textheight}%
}{0.5ex}}%
\stackon[1pt]{#1}{\scalebox{-1}{\tmpbox}}%
}
\stackMath
\newcommand\newhat[1]{%
\savestack{\tmpbox}{\stretchto{%
  \scaleto{%
    \scalerel*[\widthof{\ensuremath{#1}}]{\kern-.6pt\bigwedge\kern-.6pt}%
    {\rule[-\textheight/2]{1ex}{\textheight}}
  }{\textheight}%
}{0.5ex}}%
\stackon[1pt]{#1}{\scalebox{1}{\tmpbox}}%
}

\def\cA{\mathscr{A}}

\def\cE{\mathscr{E}}

\def\cX{\mathscr{X}}
\def\cY{\mathscr{Y}}

\def\BC{\mathbb{C}}

\def\BF{\mathbb{F}}
\def\BG{\mathbb{G}}
\def\BH{\mathbb{H}}

\def\BK{\mathbb{K}}

\def\BS{\mathbb{S}}

\def\BZ{\mathbb{Z}}

\def\Bk{\mathbbm{k}}

\def\fr{\mathfrak{r}}

\mathchardef\mhyphen="2D

\def\add{\operatorname{add}}

\def\adots{\mathinner{\mkern1mu\raise1.0pt\vbox{\kern7.0pt\hbox{.}}\mkern2mu\raise5.0pt\hbox{.}\mkern2mu\raise9.0pt\hbox{.}\mkern1mu}}

\def\Ch{\operatorname{Ch}}

\def\Coker{\operatorname{Coker}}

\def\Cot{\operatorname{Cot}}

\def\dddots{\mathinner{\mkern1mu\raise10.0pt\vbox{\kern7.0pt\hbox{.}}\mkern2mu\raise5.3pt\hbox{.}\mkern2mu\raise1.0pt\hbox{.}\mkern1mu}}
\def\dddotssmall{\mathinner{\mkern1mu\raise7.0pt\vbox{\kern7.0pt\hbox{.}}\mkern-1mu\raise4pt\hbox{.}\mkern-1mu\raise1.0pt\hbox{.}\mkern1mu}}

\def\Ext{\operatorname{Ext}}

\def\Flat{\operatorname{Flat}}

\def\GInj{\operatorname{GInj}}

\def\H{\operatorname{H}}

\def\Hom{\operatorname{Hom}}
\def\id{\operatorname{id}}

\def\Inj{\operatorname{Inj}}

\def\K0{\operatorname{K}_0}

\def\Ker{\operatorname{Ker}}

\def\Mod{\operatorname{Mod}}

\def\opp{\operatorname{op}}

\def\prj{\operatorname{prj}}
\def\Prj{\operatorname{Prj}}

\def\PSL2{\operatorname{PSL}_2}

\def\SL2{\operatorname{SL}_2}

\newcommand\Tensor[1]{\,{\underset{#1}{\otimes}}\,}

\def\SerreFunctor{\BS}

\newcommand\JFunctorHomAngles[1]{{}_{#1}\Delta}
\newcommand\JFunctorHomBrackets[1]{{}_{#1}\widetilde{\Delta}}
\newcommand\LFunctorHomAngles[1]{\nabla_{\!#1}}
\newcommand\LFunctorHomBrackets[1]{\widetilde{\nabla}_{\!#1}}
\newcommand\MFunctorHomAngles[1]{\Delta_{#1}}
\newcommand\MFunctorHomBrackets[1]{\widetilde{\Delta}_{#1}}

\def\Qdiff{Q^{\operatorname{diff}}}
\def\Qcpx{Q^{\operatorname{cpx}}}
\def\QNcpx{Q^{N\!\operatorname{-cpx}}}
\def\Q3cpx{Q^{3\!\operatorname{-cpx}}}

\numberwithin{equation}{section}

\newtheorem{Lemma}{Lemma}[section]
\newtheorem{Theorem}[Lemma]{Theorem}
\newtheorem{Proposition}[Lemma]{Proposition}

\theoremstyle{definition}
\newtheorem{Definition}[Lemma]{Definition}
\newtheorem{Setup}[Lemma]{Setup}

\newtheorem{Remark}[Lemma]{Remark}

\newtheorem{Example}[Lemma]{Example}

\theoremstyle{theorem}
\newtheorem{ThmIntro}{Theorem}

\theoremstyle{definition}

\newtheorem*{bfhpg*}{}

\usepackage{amsthm}
\usepackage{thmtools}
\declaretheoremstyle[
notefont=\bfseries, notebraces={}{},
bodyfont=\normalfont,
headformat=\NUMBER~\NOTE,
headpunct={}
]{foobar}

  {\begin{list}{}{%
    \settowidth{\labelwidth}{\textbf{#1:}}%
    \setlength{\leftmargin}{\labelwidth}\addtolength{\leftmargin}{\labelsep}}}%
  {\end{list}}

\makeatletter
\@namedef{subjclassname@2020}{%
  \textup{2020} Mathematics Subject Classification}
\makeatother

\begin{document}

\setlength{\parindent}{0pt}
\setlength{\parskip}{7pt}

\title[Acyclic objects in the $Q$-shaped derived categories]{Acyclic and totally acyclic objects in the $Q$-shaped derived category}

\author{Henrik Holm}

\address{Department of Mathematical Sciences, Universitetsparken 5, University of Copenhagen, 2100 Copenhagen {\O}, Denmark} 
\email{holm@math.ku.dk}

\urladdr{http://www.math.ku.dk/\~{}holm/}

\author{Peter J\o rgensen}

\address{Department of Mathematics, Aarhus University, Ny Munkegade 118, 8000 Aarhus C, Denmark}
\email{peter.jorgensen@math.au.dk}

\urladdr{https://sites.google.com/view/peterjorgensen}




\begin{abstract} 

This paper concerns homological algebra based on $Q$-shaped diagrams, where $Q$ belongs to a certain class of small categories.  We will show that the Gorenstein property of noetherian commutative rings is characterised by the condition that acyclicity coincides with total acyclicity for $Q$-shaped diagrams of suitable classes of modules.  Many special cases occur as corollaries, for instance $N$-complexes.

\medskip
\noindent
This generalises a classic result by Iyengar and Krause to the programme of $Q$-shaped derived categories, which builds on an insight of Iyama and Minamoto.

\medskip
\noindent
We prove our result using an adjoint triple of functors relating $Q$-shaped diagrams to chain complexes.  One of the functors was introduced by Jasso, and the whole triple generalises the compression, cocompression, and expansion functors for differential modules introduced by Avramov, Buchweitz, and Iyengar and by Nkansah.  

\medskip
\noindent
We consider the adjoint triple to be the main contribution of this paper. 

\end{abstract}

\maketitle

\setcounter{section}{-1}
\section{Introduction}
\label{sec:introduction}

Recall that a chain complex over a ring $A$ is called acyclic if its homology is zero.  There is another notion, total acyclicity, which is in general stronger.  It is a classic result, originally due to Krause and Iyengar, that the Gorenstein property of noetherian commutative rings is characterised by the condition that acyclicity coincides with total acyclicity for chain complexes over suitable classes of modules; see \cite[cor.\ 5.5]{Iyengar-Krause-acyclicity} and \cite[thm.\ 19.5.25]{CFH-book}.  

We will generalise this to $Q$-shaped diagrams.  

As an avatar, consider the special $Q$-shaped diagrams known as $N$-complexes.  They were introduced by Kapranov in \cite{Kapranov} and have the same shape as chain complexes, but the relation $\partial^2 = 0$ has been replaced by $\partial^N = 0$.  The notion of $N$-acyclic (that is, $N$-exact) $N$-complexes was defined in \cite[def.\ 1.1]{Kapranov}, and conditions (ii), (iii), and (iv) of the following theorem require $N$-acyclicity to coincide with total $N$-acyclicity for $N$-complexes over projective, injective, and flat modules.

\begin{ThmIntro}
\label{thm:A}
Let $A$ be a noetherian commutative ring.  The following conditions are equivalent:
\begin{enumerate}
\setlength\itemsep{4pt}

  \item  $A$ is a Gorenstein ring.
  
  \item  If $P$ is an $N$-acyclic $N$-complex of projective $A$-modules and $L$ is a projective $A$-module, then $\Hom_A( P,L )$ is $N$-acyclic.

  \item  If $I$ is an $N$-acyclic $N$-complex of injective $A$-modules and $J$ is an injective $A$-module, then $\Hom_A( J,I )$ is $N$-acyclic.

  \item  If $F$ is an $N$-acyclic $N$-complex of flat $A$-modules and $J$ is an injective $A$-module, then $J \otimes_A F$ is $N$-acyclic.
\hfill $\Box$

\end{enumerate}
\end{ThmIntro}

If $N=2$ then Theorem \ref{thm:A} is precisely the classic result mentioned above.

Theorem \ref{thm:A} is a (very) special case of Theorem \ref{thm:B} below, which deals with general $Q$-shaped diagrams and belongs to the programme of $Q$-shaped derived categories.  That programme was established in \cite{HJ-JLMS}, \cite{HJ-TAMS} based on an insight of Iyama and Minamoto \cite{Iyama-Minamoto-1}, \cite{Iyama-Minamoto-2}; see \cite{HJ-Abel} for a survey and \cite{Di-Li-Liang-Ma_Flat-model-structures-and-Gorenstein-objects-in-functor-categories}, \cite{Di-Li-Liang-Ma_Flat-models-for-Q-shaped -derived -categories-via-PGF-objects}, \cite{HJ-Minimal}, \cite{Jasso-Q-shaped}, and \cite{Slaftsos-Vitoria} for other recent contributions.

Let $\Bk$ be a noetherian hereditary commutative ring and let $Q$ be a small $\Bk$-preadditive category satisfying conditions (Hom finiteness), (Local boundedness), (Serre functor), (Strong retraction), and (Nilpotence), which will be stated in Section \ref{subsec:notation}.  If $A$ is a $\Bk$-algebra, then a $Q$-shaped diagram of left $A$-modules is a $\Bk$-linear functor $Q \xrightarrow{} {}_{ A }\!\Mod$ to the category of left $A$-modules.  There is a notion of acyclic (also known as exact) $Q$-shaped diagram, see Equation \eqref{equ:exact_objects}, and conditions (ii), (iii), and (iv) of the following theorem require acyclicity to coincide with total acyclicity for $Q$-shaped diagrams over projective, injective, and flat modules.  

\begin{ThmIntro}
\label{thm:B}
Let $A$ be a noetherian commutative $\Bk$-algebra.  Assume that there exists $q \in Q_0$ such that $\SerreFunctor q \neq q$, where $Q_0$ is the set of objects of $Q$ and $\SerreFunctor$ is the Serre functor of $Q$.  Then the following conditions are equivalent:
\begin{enumerate}
\setlength\itemsep{4pt}

  \item  $A$ is a Gorenstein ring.
  
  \item  If $P$ is an acyclic $Q$-shaped diagram of projective $A$-modules and $L$ is a projective $A$-module, then $\Hom_A( P,L )$ is acyclic.

  \item  If $I$ is an acyclic $Q$-shaped diagram of injective $A$-modules and $J$ is an injective $A$-module, then $\Hom_A( J,I )$ is acyclic.

  \item  If $F$ is an acyclic $Q$-shaped diagram of flat $A$-modules and $J$ is an injective $A$-module, then $J \otimes_A F$ is acyclic.
\hfill $\Box$

\end{enumerate}
\end{ThmIntro}

Theorem \ref{thm:B} applies to many examples, among them the following:
\begin{itemize}
\setlength\itemsep{4pt}

  \item  $Q$ can be $\QNcpx$, the $\Bk$-linear category given by the quiver in Figure \ref{fig:linear_quiver} with $N \geqslant 2$ and the relations that $N$ consecutive arrows compose to zero.  Then the $Q$-shaped diagrams are the $N$-complexes, and Theorem \ref{thm:B} specialises to Theorem \ref{thm:A}.

  \item  $Q$ can be the $\Bk$-linear category given by the quiver in Figure \ref{fig:cyclic_quiver} with $m \geqslant 2$ and the relations that two consecutive arrows compose to zero.  Then the $Q$-shaped diagrams are the $m$-periodic chain complexes.
  
  \item  $Q$ can be the $\Bk$-linear category given by the repetitive quiver $\BZ A_n$ with $n \geqslant 2$ and mesh relations.  This is a less classic situation.
  
  \item  If $\Bk$ is a field, $\Lambda$ a finite dimensional $\Bk$-algebra which is not weakly symmetric in the sense of \cite[p.\ 378]{Skowronski-Yamagata_Frobenius-algebras-I}, then $Q$ can be the $\Bk$-linear category given by the Gabriel quiver of $\Lambda$.  Then the $Q$-shaped diagrams are the left $\Lambda^{ \opp } \Tensor{\Bk} A$-modules.
  
\smallskip
\noindent  
The condition that $\Lambda$ is not weakly symmetric is imposed because it implies that there exists $q \in Q_0$ such that $\SerreFunctor q \neq q$.

\end{itemize}
See also \cite[2.5]{HJ-Abel}.

Theorem \ref{thm:B} will be proved using the adjoint triple of functors shown in Figure \ref{fig:JLM}.  They relate $Q$-shaped diagrams to chain complexes, and we view them as the main contribution of this paper.  We will only establish the properties necessary for the proof of Theorem \ref{thm:B}, in which the functors are used to translate the classic result from chain complexes to $Q$-shaped diagrams.

However, more could be said, and we consider it likely that the functors also have other applications.

The functor $\LFunctorHomAngles{q}$ in Figure \ref{fig:JLM} is defined using a complete projective resolution $P\langle q \rangle$ of a certain Gorenstein projective object $S\langle q \rangle$.
The definition is due to Jasso, see \cite[sec.\ 2.4.3]{Jasso-Q-shaped}.  The left and right adjoints $\JFunctorHomAngles{q}$ and $\MFunctorHomAngles{q}$ have not appeared previously.  However, one special case of Figure \ref{fig:JLM} does play a crucial role in the existing literature on differential modules, where the functors occur as the compression, cocompression, and expansion functors introduced by Avramov, Buchweitz, and Iyengar and by Nkansah; see \cite[sec.\ 1]{Avramov-Buchweitz-Iyengar} and \cite[sec.\ 2]{Nkansah-differential-modules}.  This will be explained in Section \ref{subsec:differential_modules}.

The paper is organised as follows: Section \ref{sec:preliminary} contains preliminary lemmas belonging to the programme of $Q$-shaped derived categories, which is the ambience for everything we do.  Section \ref{sec:complete} constructs a complete projective resolution which can be used to define $\LFunctorHomAngles{q}$.  Sections \ref{sec:L} and \ref{sec:JM} define $\LFunctorHomAngles{q}$, $\JFunctorHomAngles{q}$, and $\MFunctorHomAngles{q}$ and establish some of their properties.  Section \ref{sec:proof_of_Thm_B} proves Theorems \ref{thm:A} and \ref{thm:B}.

\begin{figure}
\begin{tikzpicture}[scale=1.5]
  \node at (-3,0){$\cdots$};
  \draw[->] (-2.75,0) to (-2.2,0);
  \node at (-2,0){\textsf{2}};
  \draw[->] (-1.8,0) to (-1.2,0);  
  \node at (-1,0){\textsf{1}};
  \draw[->] (-0.8,0) to (-0.2,0);    
  \node at (0,0){\textsf{0}};
  \draw[->] (0.2,0) to (0.75,0);    
  \node at (1,0){\textsf{-1}};
  \draw[->] (1.25,0) to (1.75,0);    
  \node at (2,0){\textsf{-2}};    
  \draw[->] (2.25,0) to (2.75,0);    
  \node at (3,0){$\cdots$};    
\end{tikzpicture}
\caption{Chain complexes and $N$-complexes are diagrams of this form.}
\label{fig:linear_quiver}
\end{figure}
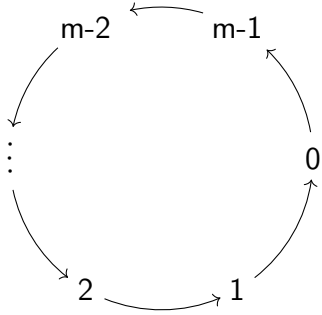
\begin{figure}
\begin{tikzpicture}[scale=2]
  \node at (0:1.0){\textsf{0}};
  \draw[->] (10:1.0) arc (10:46:1.0);
  \node at (60:1.0){\textsf{m-1}};
  \draw[->] (74:1.0) arc (74:102:1.0);
  \node at (120:1.0){\textsf{m-2}};
  \draw[->] (133:1.0) arc (133:168:1.0);
  \node at (175:1.0){$\cdot$};
  \node at (180:1.0){$\cdot$};
  \node at (185:1.0){$\cdot$};  
  \draw[->] (192:1.0) arc (192:232:1.0);
  \node at (240:1.0){\textsf{2}};
  \draw[->] (248:1.0) arc (248:293:1.0);
  \node at (300:1.0){\textsf{1}};
  \draw[->] (308:1.0) arc (308:352:1.0);
\end{tikzpicture}
\caption{$m$-periodic chain complexes and $m$-periodic $N$-complexes are diagrams of this form.}
\label{fig:cyclic_quiver}
\end{figure}
\begin{figure}
\[
\xymatrix
{
  {}_{ Q,A }\!\Mod
    \ar[rr]^{ \LFunctorHomAngles{q} } &&
  \Ch( {}_{ A }\!\Mod ),
    \ar@/_2.0pc/[ll]_{ \JFunctorHomAngles{q} }
    \ar@/^2.0pc/[ll]^{ \MFunctorHomAngles{q} }
}
\]
\caption{The main contribution of this paper is an adjoint triple of functors relating $Q$-shaped diagrams of $A$-modules (left) to chain complexes of $A$-modules (right).}
\label{fig:JLM}
\end{figure}

We now explain the adjoint triple in Figure \ref{fig:JLM} in two examples for which we need the following definitions:
\begin{align*}
  {}_{ Q }\!\Mod & = \{ \Bk\mbox{-linear functors }Q \xrightarrow{} {}_{ \Bk }\!\Mod \}, \\
  {}_{ Q,A }\!\Mod & = \{ \Bk\mbox{-linear functors }Q \xrightarrow{} {}_{ A }\!\Mod \}.
\end{align*}

\subsection{Example: Differential modules}
\label{subsec:differential_modules}
In this subsection, which is due to \cite{Avramov-Buchweitz-Iyengar} and \cite{Nkansah-differential-modules}, we let $Q$ be the $\Bk$-preadditive category $\Qdiff$ given by
\[
  \xymatrix{
    q \ar@(ul,ur)^{ \delta }
           }
\]
with $\delta^2 = 0$.

Then ${}_{ Q }\!\Mod$ and ${}_{ Q,A }\!\Mod$ are the categories of differential modules over $\Bk$ and $A$, where a differential module is a pair $( X,\partial^X )$ with $X$ in ${}_{ \Bk }\!\Mod$ or ${}_{ A }\!\Mod$ and $\partial^X$ is an endomorphism with square zero.  We sometimes abbreviate $( X,\partial^X )$ to $X$.

There is a ``pseudosimple'' differential module $S\langle q \rangle = ( \Bk,0 )$.  If $\Bk$ is a field, then $S\langle q \rangle$ is genuinely simple in ${}_{ Q }\!\Mod$.

There is also a differential module $F = ( \Bk \oplus \Bk,\partial^F )$ where $d^F\begin{bsmallmatrix} 1 \\ 0 \end{bsmallmatrix} = \begin{bsmallmatrix} 0 \\ 1 \end{bsmallmatrix}$ and $d^F\begin{bsmallmatrix} 0 \\ 1 \end{bsmallmatrix} = \begin{bsmallmatrix} 0 \\ 0 \end{bsmallmatrix}$.  It is a projective object of ${}_{ Q }\!\Mod$.

The differential module $S\langle q \rangle$ is a Gorenstein projective object of ${}_{ Q }\!\Mod$, and it has a complete projective resolution 
$
  P\langle q \rangle =
  \cdots
  \xrightarrow{}
  F
  \xrightarrow{}
  F
  \xrightarrow{}
  F
  \xrightarrow{}
  \cdots.
$
We define the functor $\LFunctorHomAngles{q}$ in Figure \ref{fig:JLM} by
\[
  \LFunctorHomAngles{q}( - ) = \Hom_Q( P\langle q \rangle,- ).
\]  
It expands a differential module $X$ to the $1$-periodic chain complex
\[
  \LFunctorHomAngles{q}X 
  = 
  \cdots
  \xrightarrow{}
  X
  \xrightarrow{\partial^X}
  X
  \xrightarrow{\partial^X}
  X
  \xrightarrow{}
  \cdots.
\]
Hence $\LFunctorHomAngles{q}$ is the expansion functor introduced in \cite[1.4]{Avramov-Buchweitz-Iyengar}.  

The left adjoint $\JFunctorHomAngles{q}$ compresses a chain complex $C$ to the differential module
\[
  \JFunctorHomAngles{q}C
  =
  \Big(
  \coprod_i C_i,\coprod_i \partial^C_i
  \Big).
\]
Hence $\JFunctorHomAngles{q}$ is the compression functor introduced in \cite[1.3]{Avramov-Buchweitz-Iyengar}.  

The right adjoint $\MFunctorHomAngles{q}$  ``cocompresses'' a chain complex $C$ to the differential module
\[
  \MFunctorHomAngles{q}C
  =
  \Big(
  \prod_i C_i,\prod_i \partial^C_i
  \Big).
\]
Hence $\MFunctorHomAngles{q}$ is the cocompression functor introduced in \cite[sec.\ 1]{Nkansah-differential-modules}.

\subsection{Example: $3$-complexes}
\label{subsec:3-complexes}
In this subsection we let $Q$ be the $\Bk$-preadditive category $\Q3cpx$ given by the quiver in Figure \ref{fig:linear_quiver} with the relations that three consecutive arrows compose to zero.  A sans serif font is used to denote the vertices of the quiver, which are also the objects of $Q$.

Then ${}_{ Q }\!\Mod$ and ${}_{ Q,A }\!\Mod$ are the categories of $3$-complexes over $\Bk$ and $A$ in the sense of \cite[def.\ 0.1]{Kapranov}.

For each $q \in Q_0$, there is a ``pseudosimple'' $3$-complex $S\langle q \rangle = \cdots \xrightarrow{} 0 \xrightarrow{} \Bk \xrightarrow{} 0 \xrightarrow{} \cdots$ with $\Bk$ placed at $q$.  If $\Bk$ is a field, then $S\langle q \rangle$ is genuinely simple in ${}_{ Q }\!\Mod$.

For each $q \in Q_0$, there is a $3$-complex $F_q = \cdots \xrightarrow{} 0 \xrightarrow{} \Bk \xrightarrow{ \id } \Bk \xrightarrow{ \id } \Bk \xrightarrow{} 0 \xrightarrow{} \cdots$ with the first $\Bk$ placed at $q$.  It is a projective object of ${}_{ Q }\!\Mod$.

Each $S\langle q \rangle$ is a Gorenstein projective object of ${}_{ Q }\!\Mod$.  This is in particular true of $S \langle \textsf{0} \rangle$, which has a complete projective resolution 
$
  P\langle \textsf{0} \rangle =
  \cdots
  \xrightarrow{}
  F_{ \textsf{-3} }
  \xrightarrow{}
  F_{ \textsf{-1} }
  \xrightarrow{}
  F_{ \textsf{0} }
  \xrightarrow{}
  F_{ \textsf{2} }
  \xrightarrow{}
  F_{ \textsf{3} }
  \xrightarrow{}
  \cdots.
$
We define the functor $\LFunctorHomAngles{ \textsf{0} }$ in Figure \ref{fig:JLM} by
\[
  \LFunctorHomAngles{ \textsf{0} }( - ) = \Hom_Q( P\langle \textsf{0} \rangle,- ).
\]  
It sends a $3$-complex $X = \cdots \xrightarrow{}X_{ \textsf{2} } \xrightarrow{} X_{ \textsf{1} } \xrightarrow{} X_{ \textsf{0} } \xrightarrow{}  X_{ \textsf{-1} } \xrightarrow{} X_{ \textsf{-2} }  \xrightarrow{} \cdots$ to the chain complex
\[
  \LFunctorHomAngles{ \textsf{0} }X 
  = 
  \cdots
  \xrightarrow{}
  X_{ \textsf{3} }
  \xrightarrow{}
  X_{ \textsf{2} }
  \xrightarrow{}
  X_{ \textsf{0} }
  \xrightarrow{}
  X_{ \textsf{-1} }
  \xrightarrow{}
  X_{ \textsf{-3} }
  \xrightarrow{}
  \cdots
\]
with $X_{ \textsf{0} }$ in degree zero, where the differentials are given by the relevant compositions of differentials in $X$.

The left adjoint $\JFunctorHomAngles{ \textsf{0} }$ sends a chain complex $C$ to the $3$-complex
\[
  \JFunctorHomAngles{ \textsf{0} }C
  =
  \cdots
  \xrightarrow{}
  C_{ 2 }
  \xrightarrow{}
  C_{ 1 }
  \xrightarrow{}
  C_{ 1 }
  \xrightarrow{}
  C_{ 0 }
  \xrightarrow{}
  C_{ -1 }
  \xrightarrow{}
  C_{ -1 }
  \xrightarrow{}
  C_{ -2 }
  \xrightarrow{}
  \cdots
\]
with $C_0$ at $\textsf{0}$, and the right adjoint $\MFunctorHomAngles{ \textsf{0} }$ sends a chain complex $C$ to the $3$-complex
\[
  \MFunctorHomAngles{ \textsf{0} }C
  =
  \cdots
  \xrightarrow{}
  C_{ 2 }
  \xrightarrow{}
  C_{ 1 }
  \xrightarrow{}
  C_{ 0 }
  \xrightarrow{}
  C_{ 0 }
  \xrightarrow{}
  C_{ -1 }
  \xrightarrow{}
  C_{ -2 }
  \xrightarrow{}
  C_{ -2 }
  \xrightarrow{}
  \cdots
\]
with $C_0$ at $\textsf{0}$; the differentials are either identity maps or differentials in $C$.

If, instead of starting from $S\langle \textsf{0} \rangle$, we start from $S\langle q \rangle$, then we obtain the functors $\LFunctorHomAngles{q}$, $\JFunctorHomAngles{q}$, and $\MFunctorHomAngles{q}$ where everything has been translated suitably.

\subsection{Notation}
\label{subsec:notation}

Throughout the paper, the following objects are fixed:
\begin{itemize}
\setlength\itemsep{4pt}

  \item  $\Bk$ is a noetherian hereditary commutative ring.
  
  \item  $A$ is a $\Bk$-algebra.  Note that any ring can be used as $A$ if $\Bk$ is equal to the integers $\BZ$.

  \item  $Q$ is a small $\Bk$-preadditive category, that is, a category where each $\Hom$ space is a $\Bk$-module and composition distributes over addition.  We assume $Q$ has the following properties, where $Q_0$ is the set of objects of $Q$ and $Q( p,q )$ is the $\Hom$ space in $Q$ from $p$ to $q$.
\medskip
\begin{itemize}
\setlength\itemsep{4pt}

  \item  (Hom finiteness):  Each $Q( p,q )$ is a finitely generated projective $\Bk$-module.
  
  \item  (Local boundedness):  For each $q \in Q_0$ the following sets are finite:
\begin{align*}
  N_-( q ) & = \{\, p \in Q_0 \mid Q( p,q ) \neq 0 \,\}, \\
  N_+( q ) & = \{\, p \in Q_0 \mid Q( q,p ) \neq 0 \,\}.
\end{align*}

  \item  (Serre functor):  There is a $\Bk$-linear autoequivalence $\SerreFunctor: Q \xrightarrow{} Q$ such that there are isomorphisms
\[  
  Q( q,\SerreFunctor p ) \cong \Hom_{ \Bk }\!\big( Q( p,q ),\Bk \big),
\]  
natural with respect to $p, q \in Q_0$.  

  \item  (Strong retraction):  For each $q \in Q_0$ the $\Bk$-module $Q( q,q )$ is equipped with a decomposition
\[
  Q( q,q ) = \Bk \cdot \id_q \oplus\, \fr_q,
\]
and the decompositions satisfy the following:
\medskip
\begin{enumerate}
\setlength\itemsep{4pt}

  \item  $\fr_q \circ \fr_q \subseteq \fr_q$ for $q \in Q_0$,
  
  \item  $Q( p,q ) \circ Q( q,p ) \subseteq \fr_q$ for $p \neq q$ in $Q_0$.

\end{enumerate}
We define a subfunctor $\fr( -,- )$ of $Q( -,- )$, called the pseudoradical, by
\[
  \fr( p,q ) = 
  \left\{
    \begin{array}{cl}
      Q( p,q ) & \mbox{if $p \neq q$,} \\[1mm]
      \fr_q    & \mbox{if $p = q$.} \\
    \end{array}
  \right.
\]

  \item  (Nilpotence):  Abbreviating $\fr( -,- )$ to $\fr$, there is an integer $N \geqslant 0$ such that $\fr^N = 0$. 

\end{itemize}
\end{itemize}
The opposite category $Q^{ \opp }$ also has these properties.  Its pseudoradical is given by $\fr^{ \opp }( p,q ) = \fr( q,p )$, and its Serre functor is $( \SerreFunctor^{ -1 } )^{ \opp }$, that is, the inverse $\SerreFunctor^{ -1 }$ viewed as a functor $Q^{ \opp } \xrightarrow{} Q^{ \opp }$; see \cite[p.\ 33]{MacLane-book}.

For the convenience of the reader, we now collect a list of notation used in the rest of the paper.  It mainly matches our other papers on $Q$-shaped derived categories, see \cite{HJ-Abel}, \cite{HJ-JLMS}, \cite{HJ-TAMS}.  

Note that if $Q$ is $\Qcpx$, the $\Bk$-linear category given by the quiver in Figure \ref{fig:linear_quiver} with the relations that two consecutive arrows compose to zero, then a $Q$-shaped diagram is a chain complex, and the items in the list specialise to the relevant classic items.  In particular, acyclic objects are chain complexes which are acyclic in the classic sense of having zero homology. 
\begin{itemize}
\setlength\itemsep{4pt}

  \item  Opposite rings and categories are denoted by the superscript ``$\opp$''.

  \item  We freely view contravariant functors as covariant functors on opposite categories and vice versa.

  \item  An additive subcategory is a subcategory closed under isomorphisms, finite coproducts, and direct summands.

  \item  The category of $\Bk$-modules is ${}_{ \Bk }\!\Mod$, and its $\Hom$ functor is $\Hom_{ \Bk }$.

  \item  The category of left $A$-modules is ${}_{ A }\!\Mod$, and its $\Hom$ functor is $\Hom_A$.

  \item  The category of right $A$-modules is $\Mod_A$, and its $\Hom$ functor is $\Hom_{ A^{ \opp } }$.

  \item  The category of $A$-bimodules is ${}_{ A }\!\Mod_A$.
  
  \item  The following categories of $Q$-shaped diagrams will be used:
\begin{align*}
  {}_{ Q }\!\Mod & = \{ \mbox{$\Bk$-linear functors $Q \xrightarrow{} {}_{ \Bk }\!\Mod$} \} && \hspace{-3em}\mbox{with $\Hom$ functor $\Hom_Q$}. \\
  \Mod_Q & = \{ \mbox{$\Bk$-linear functors $Q^{ \opp } \xrightarrow{} {}_{ \Bk }\!\Mod$} \} && \hspace{-3em}\mbox{with $\Hom$ functor $\Hom_{ Q^{ \opp } }$}. \\
  {}_{ Q,A }\!\Mod & = \{ \mbox{$\Bk$-linear functors $Q \xrightarrow{} {}_{ A }\!\Mod$} \} && \hspace{-3em}\mbox{with $\Hom$ functor $\Hom_{ Q,A }$}. \\
  \Mod_{ Q,A } & = \{ \mbox{$\Bk$-linear functors $Q^{ \opp } \xrightarrow{} \Mod_A$} \} && \hspace{-3em}\mbox{with $\Hom$ functor $\Hom_{ Q^{ \opp },A^{ \opp } }$}.
\end{align*}

  \item  For each of the categories above, replacing ``$\Mod$'' by ``$\Prj$'' or ``$\Inj$'' denotes the full subcategory of projective or injective objects.
  
  \item  The category of finitely generated projective $\Bk$-modules is ${}_{ \Bk }\!\prj$. 

  \item  There are full subcategories of acyclic objects (called exact in our other papers):
\begin{equation}
\label{equ:exact_objects}
  \begin{array}{rcl}
    {}_{ Q }\cE & = & \{\, X \in {}_{ Q }\!\Mod \,|\, \mbox{$X$ has finite projective dimension in ${}_{ Q }\!\Mod$} \,\}, \\[2mm]
    {}_{ Q,A }\cE & = & \{\, Y \in {}_{ Q,A }\!\Mod \,|\, \mbox{$Y^{ \natural }$ has finite projective dimension in ${}_{ Q }\!\Mod$} \,\}, \\
  \end{array}
\end{equation}
where $Y^{ \natural }$ denotes the object of ${}_{ Q }\!\Mod$ obtained by forgetting the $A$-structure of $Y\!$.  Symmetrically, there are full subcategories $\cE_Q$ and $\cE_{ Q,A }$.  See \cite[def.\ 4.1]{HJ-JLMS}.

  \item  $\Ext$ functors are decorated with the same subscripts used for the corresponding $\Hom$ functors.
  
  \item  There is a tensor functor $- \Tensor{Q} - : \Mod_Q \times {}_{ Q }\!\Mod \xrightarrow{} {}_{ \Bk }\!\Mod$, see \cite[sec.\ 1]{Oberst-Roehrl}.  
  
  \item  $\Hom$, $\Ext$, and tensor functors are compatible with additional structures.  A specific instance is particularly important in this paper: Given $P \in {}_{ Q }\!\Mod$ and $X \in {}_{ Q,A }\!\Mod$, we have $\Hom_Q( P,X ) \in {}_{ A }\!\Mod$.
  
  \item  For each $q \in Q_0$ there is a ``pseudosimple'' functor $S\langle q \rangle = Q( q,- )/\fr( q,- )$ in ${}_{ Q }\!\Mod$.  It satisfies
\begin{equation}
\label{equ:S_formulae_0}
  S\langle q \rangle( p )
  =
  \left\{
    \begin{array}{cl}
      \Bk & \mbox{if $p = q$,} \\[1mm]
      0   & \mbox{if $p \neq q$;} \\
    \end{array}  
  \right.
\end{equation}
see \cite[def.\ 7.9 and lem.\ 7.10]{HJ-JLMS}.  Dually, there is $S\{ q \} = Q( -,q )/\fr( -,q )$ in $\Mod_Q$, and we have
\begin{equation}
\label{equ:S_formulae_1}
  S\langle q \rangle \cong \Hom_{ \Bk }( S\{ q \},\Bk )
  \;\;\mbox{and}\;\;
  S\{ q \} \cong \Hom_{ \Bk }( S\langle q \rangle,\Bk ).
\end{equation}
See \cite[prop.\ 7.15 and its proof]{HJ-JLMS}.

  \item  For each $q \in Q_0$ there is a functor 
\[
  S_q : {}_{ A }\!\Mod \xrightarrow{} {}_{ Q,A }\!\Mod
\]
given by
\begin{equation}
\label{equ:S_formulae_2}
  S_q( - )
  = S\langle q \rangle \Tensor{\Bk} -
  \cong \Hom_{ \Bk }( S\{ q \}, - ),
\end{equation}
which places an $A$-module $M$ at $q$ and puts zeroes elsewhere.  See \cite[prop.\ 7.15 and its proof]{HJ-JLMS}.

  \item  For each $q \in Q_0$ there is an adjoint triple
\begingroup
\begin{equation}
\label{equ:EFG}
\xymatrix
{
  {}_{ Q,A }\!\Mod
    \ar[rr]^{ E_q } &&
  {}_{ A }\!\Mod
    \ar@/_2.0pc/[ll]_{ F_q }
    \ar@/^2.0pc/[ll]^{ G_q }
}
\;\;\; \mbox{ given by } \;\;\;
\left\{
  \begin{array}{l}
    F_qM = Q( q,- ) \Tensor{\Bk} M, \\
    E_qX = X( q ), \\[2mm]
    G_qM = \Hom_{ \Bk }\!\big( Q( -,q ),M \big),
  \end{array}
\right.
\end{equation}
\endgroup
see \cite[cor.\ 3.9]{HJ-JLMS}.

  \item  If $\cA$ is an additive category, then $\Ch( \cA )$ is the category of chain complexes over $\cA\!$.  Its suspension functor is denoted $\Sigma$.

\end{itemize}

\section{Preliminary results}
\label{sec:preliminary}

This section contains preliminary lemmas belonging to the programme of $Q$-shaped derived categories, which is the ambience for everything we do.

\begin{Definition}
\label{def:2}
By the proof of \cite[thm.\ 7.29]{HJ-JLMS}, the full subcategory ${}_{ Q }\!\Prj$ of projective objects of ${}_{ Q }\!\Mod$ can be characterised as follows:

Let $V$ be a subset of $Q_0$ and let $K_q$ be in ${}_{ \Bk }\!\Mod$ for each $q \in V$.  Then
\begin{equation}
\label{equ:def:2:10}
  \coprod_{ q \in V } Q( q,- ) \Tensor{\Bk} K_q
\end{equation}
is an object of ${}_{ Q }\!\Mod$, and 
\begin{itemize}
\setlength\itemsep{4pt}

  \item  ${}_{ Q }\!\Prj$ is the full subcategory of objects \eqref{equ:def:2:10} where $V$ is an arbitrary subset of $Q_0$ and each $K_q$ is in ${}_{ \Bk }\!\Prj$.

\end{itemize}
We also make the following definition:
\begin{itemize}
\setlength\itemsep{4pt}

  \item  ${}_{ Q }\!\prj$ is the full subcategory of objects \eqref{equ:def:2:10} where $V$ is a finite subset of $Q_0$ and each $K_q$ is in ${}_{ \Bk }\!\prj$.
  
\end{itemize}
Dually,
\[
  \coprod_{ q \in V } Q( -,q ) \Tensor{\Bk} K_q
\]
is an object of $\Mod_Q$, and we use these as above to characterise the subcategory $\Prj_Q$ and define the subcategory $\prj_Q$ of $\Mod_Q$.

In particular, ${}_{ Q }\!\prj$, respectively $\prj_Q$, consists of objects which are projective in ${}_{ Q }\!\Mod$, respectively $\Mod_Q$.
\end{Definition}

\begin{Lemma}
\label{lem:pre_5i_5ii_30_32}
Let $P \in {}_{ Q }\!\prj$ be given.  
\begin{enumerate}
\setlength\itemsep{4pt}

  \item  The functor
\[
  \Hom_Q( P,? ) : {}_{ Q }\!\Mod \xrightarrow{} {}_{ \Bk }\!\Mod
\]
respects set indexed products and coproducts.

  \item  Let $\cA \subseteq {}_{ A }\!\Mod$ be an additive subcategory and let $X \in {}_{ Q,A }\!\Mod$ be given.  Then the following implication holds:
\[
  X( q ) \in \cA \mbox{ for each } q \in Q_0 \Rightarrow \Hom_Q( P,X ) \in \cA\!.
\]

\end{enumerate}
\end{Lemma}

\begin{proof}
It is clear that $\Hom_Q( P,? )$ respects set indexed products.

Observe that $P$ has the form \eqref{equ:def:2:10} with $V$ finite and each $K_q$ in ${}_{ \Bk }\!\prj$.  This implies
\begin{align*}
  \Hom_Q( P,? )
  & = \Hom_Q\!\big( \coprod_{ q \in V } Q( q,- ) \Tensor{\Bk} K_q,? \big) && \\
  & \cong \coprod_{ q \in V } \Hom_Q\!\big( Q( q,- ) \Tensor{\Bk} K_q,? \big) && \mbox{$V$ finite} \\
  & \cong \coprod_{ q \in V } \Hom_{ \Bk }\!\Big( K_q,\Hom_Q\!\big( Q( q,- ),? \big) \Big) && \mbox{Adjointness} \\
  & \cong \coprod_{ q \in V } \Hom_{ \Bk }\!\big( K_q,(?)(q) \big). && \mbox{Yoneda's Lemma}
\end{align*}
Since $V$ is finite and each $K_q$ is in ${}_{ \Bk }\!\prj$, the computation implies that $\Hom_Q( P,? )$ respects set indexed coproducts.  It also implies part (ii) of the lemma.
\end{proof}

\begin{Lemma}
\label{lem:16A_28_35_56}
\begin{enumerate}
\setlength\itemsep{4pt}

  \item  There is an isomorphism in $\Mod_Q$,
\[
  \Hom_{ \Bk }\!\big( Q( q,- ) \Tensor{\Bk} K,\Bk  \big) \cong Q( -,\SerreFunctor q ) \Tensor{\Bk} \Hom_{ \Bk }( K,\Bk ),
\]
natural with respect to $q \in Q_0$ and $K \in {}_{ \Bk }\!\prj$.

  \item  There is an isomorphism in ${}_{ \Bk }\!\Mod$,
\[
  F\big( \Hom_Q( P,X ) \big) \cong \Hom_Q( P,F \circ X ),
\]
natural with respect to $P \in {}_{ Q }\!\prj$, $X \in {}_{ Q }\!\Mod$, and the $\Bk$-linear functor $F: {}_{ \Bk }\!\Mod \xrightarrow{} {}_{ \Bk }\!\Mod$.

  \item  There is an isomorphism in ${}_{ \Bk }\!\Mod$,
\[
  \Hom_{ Q }\!\big( X,\Hom_{ \Bk }( Y,M ) \big)
  \cong 
  \Hom_{ Q^{ \opp } }\!\big( Y,\Hom_{ \Bk }( X,M ) \big),
\]
natural with respect to $X \in {}_{ Q }\!\Mod$, $Y \in \Mod_{ Q }$, and $M \in {}_{ \Bk }\!\Mod$.

  \item  There is an isomorphism in ${}_{ \Bk }\!\Mod$,
\[
  \Hom_{ \Bk }( X,M ) \Tensor{Q} P
  \cong
  \Hom_{ \Bk }\!\big( \Hom_Q( P,X ),M \big),
\]
natural with respect to $X \in {}_{ Q }\!\Mod$, $M \in {}_{ \Bk }\!\Mod$, and $P \in {}_{ Q }\!\prj$.

  \item  There is an isomorphism in ${}_{ \Bk }\!\Mod$,
\[
  \Hom_{ Q^{ \opp } }\!\big( \Hom_{ \Bk }( P,\Bk ),Y \big) \cong ( Y \circ \SerreFunctor ) \Tensor{Q} P,
\]
natural with respect to $P \in {}_{ Q }\!\prj$ and $Y \in \Mod_Q$.

  \item  There is an isomorphism in ${}_{ \Bk }\!\Mod$,
\[
  \Hom_{ Q^{ \opp } }\!\big( \Hom_{ \Bk }( P,\Bk ),\Hom_{ \Bk }( X,M ) \circ \SerreFunctor^{ -1 } \big) \cong \Hom_{ \Bk }\!\big( \Hom_Q( P,X ),M \big),
\]
natural with respect to $P \in {}_{ Q }\!\prj$, $X \in {}_{ Q }\!\Mod$, and $M \in {}_{ \Bk }\!\Mod$.

\end{enumerate}
\end{Lemma}

\begin{proof}
(i): We can compute as follows:
\begin{align*}
  \Hom_{ \Bk }\!\big( Q( q,- ) \Tensor{\Bk} K,\Bk  \big)
  & \cong \Hom_{ \Bk }\!\big( Q( q,- ),\Hom_{ \Bk }( K,\Bk ) \big) && \mbox{Adjointness} \\
  & \cong \Hom_{ \Bk }\!\big( Q( q,- ),\Bk \Tensor{\Bk} \Hom_{ \Bk }( K,\Bk ) \big) && \\
  & \cong \Hom_{ \Bk }\!\big( Q( q,- ),\Bk \big) \Tensor{\Bk} \Hom_{ \Bk }( K,\Bk ) && \mbox{\cite[prop.\ 1.4.6(a)]{CFH-book}} \\
  & \cong Q( -,\SerreFunctor q ) \Tensor{\Bk} \Hom_{ \Bk }( K,\Bk ). && \mbox{Serre duality}
\end{align*}

(ii): We obtain the following isomorphisms from Yoneda's Lemma:
\begin{align*}
  F\Big( \Hom_Q\!\big( Q( q,- ),X \big) \Big) & \cong F\big( X( q ) \big), \\
  \Hom_Q\!\big( Q( q,- ),F \circ X \big) & \cong ( F \circ X )( q ),
\end{align*}
and the right hand sides are identical, so there is an isomorphism between the left hand sides:
\[
  F\Big( \Hom_Q\!\big( Q( q,- ),X \big) \Big) \cong \Hom_Q\!\big( Q( q,- ),F \circ X \big),
\]
natural with respect to $q \in Q_0$, $X \in {}_{ Q }\!\Mod$, and the $\Bk$-linear functor $F : {}_{ \Bk }\!\Mod \xrightarrow{} {}_{ \Bk }\!\Mod$.  This extends to the isomorphism in part (ii) of the lemma by \cite[prop.\ 2.3(b)]{AusRepI}.

(iii): Consider an element in the left hand side of the isomorphism in part (iii), that is, a natural transformation $X \xrightarrow{ \xi } \Hom_{ \Bk }( Y,M )$.  We  map it to an element in the right hand side of the isomorphism, that is, a natural transformation $Y \xrightarrow{ \upsilon } \Hom_{ \Bk }( X,M )$ as follows: For $q \in Q_0$, $x \in X( q )$, and $y \in Y( q )$, we set
\[
  \upsilon_q( y )( x ) = \xi_q( x )( y ).
\]
Straightforward computations show that this provides the isomorphism claimed in part (iii).

(iv): There are isomorphisms
\begin{align*}
  \Hom_{ \Bk }( X,M ) \Tensor{Q} Q( q,- ) & \cong \Hom_{ \Bk }( X,M )( q ) && \mbox{\cite[p.\ 93]{Oberst-Roehrl}} \\
  & = \Hom_{ \Bk }\!\big( X( q ),M \big) && \\
  & \cong \Hom_{ \Bk }\!\Big( \Hom_Q\!\big( Q( q,- ),X \big),M \Big), && \mbox{Yoneda's Lemma}
\end{align*}
natural with respect to $q \in Q_0$, $X \in {}_{ Q }\!\Mod$, and $M \in {}_{ \Bk }\!\Mod$.  The isomorphism between the first and last of these objects extends to the isomorphism in part (iv) of the lemma by \cite[prop.\ 2.3(b)]{AusRepI}.

(v): There are isomorphisms
\begin{align*}
  \Hom_{ Q^{ \opp } }\!\Big( \Hom_{ \Bk }\big( Q( q,- ),\Bk \big),Y \Big)
  & \cong \Hom_{ Q^{ \opp } }\!\big( Q( -,\SerreFunctor q ),Y \big)  && \mbox{Serre duality} \\
  & \cong Y( \SerreFunctor q ) && \mbox{Yoneda's lemma} \\
  & \cong ( Y \circ \SerreFunctor )( q ) && \\
  & \cong ( Y \circ \SerreFunctor ) \Tensor{Q} Q( q,- ), && \mbox{\cite[p.\ 93]{Oberst-Roehrl}}
\end{align*}
natural with respect to $q \in Q_0$ and $Y \in \Mod_Q$.  The isomorphism between the first and last of these objects extends to the isomorphism in part (v) of the lemma by \cite[prop.\ 2.3(b)]{AusRepI}.

(vi): We can compute as follows:
\begin{align*}
  \Hom_{ Q^{ \opp } }\!\big( \Hom_{ \Bk }( P,\Bk ),\Hom_{ \Bk }( X,M ) \circ \SerreFunctor^{ -1 } \big)
  & \cong \big( \Hom_{ \Bk }( X,M ) \circ \SerreFunctor^{ -1 } \circ \SerreFunctor \big) \Tensor{Q} P && \mbox{Part (v)} \\
  & \cong \Hom_{ \Bk }( X,M ) \Tensor{Q} P && \\
  & \cong \Hom_{ \Bk }\!\big( \Hom_Q( P,X ),M \big). && \mbox{Part (iv)} \qedhere
\end{align*}
\end{proof}

\begin{Remark}
\label{rmk:16A_28_35_56}
Lemma \ref{lem:16A_28_35_56} is compatible with additional structures.  For instance, the isomorphism in part (iii) also exists in a version
\[
  \Hom_{ Q,A }\!\big( X,\Hom_{ A^{ \opp } }( Y,M ) \big)
  \cong 
  \Hom_{ Q^{ \opp },A^{ \opp } }\!\big( Y,\Hom_A( X,M ) \big),
\]
natural with respect to $X \in {}_{ Q,A }\!\Mod$, $Y \in \Mod_{ Q^{ \opp },A^{ \opp } }$, and $M \in {}_{ A }\!\Mod_A$.
\end{Remark}

\begin{Lemma}
\label{lem:algebraic_duality}
\begin{enumerate}
\setlength\itemsep{4pt}

    \item  The assignment
\[
  ? \mapsto \Hom_{ \Bk }( ?,\Bk )
\]
defines a duality
\[
  \{\, X \in {}_{ Q }\!\Mod \mid \mbox{$X( q ) \in {}_{ \Bk }\!\prj$ for each $q \in Q_0$} \,\}
  \leftrightarrow
  \{\, Y \in \Mod_Q \mid \mbox{$Y( q ) \in {}_{ \Bk }\!\prj$ for each $q \in Q_0$} \,\}.
\]

  \item  Let $p \in Q_0$ be given.  If $P$ is in ${}_{ Q }\!\prj$ or $\prj_Q$, then $P( p )$ is in ${}_{ \Bk }\!\prj$.

  \item  The duality in part (i) restricts to a duality between ${}_{ Q }\!\prj$ and $\prj_Q$.

\end{enumerate}
\end{Lemma}

\begin{proof}
(i): This holds because $\Hom_{ \Bk }( ?,\Bk )$ is a duality from ${}_{ \Bk }\!\prj$ to itself by \cite[prop.\ 1.4.3]{CFH-book}.  

(ii): By symmetry, it is enough to consider the case where $P$ is in ${}_{ Q }\!\prj$.  But then $P$ has the form \eqref{equ:def:2:10} with $V$ finite and each $K_q$ in ${}_{ \Bk }\!\prj$, whence
\[
  P( p ) = \coprod_{ q \in V } Q( q,p ) \Tensor{\Bk} K_q
\]
is in ${}_{ \Bk }\!\prj$ because so is each $Q( q,p )$ by condition (Hom finiteness) in Section \ref{subsec:notation}.

(iii): Part (ii) says that ${}_{ Q }\!\prj$ and $\prj_Q$ are contained in the two categories in part (i), so it is enough to show that the assignment in part (i) restricts to contravariant functors ${}_{ Q }\!\prj \xrightarrow{} \prj_Q$ and $\prj_Q \xrightarrow{} {}_{ Q }\!\prj$.  By symmetry, it is enough to show the first of these claims.  Let $P$ be in ${}_{ Q }\!\prj$.  Then $P$ has the form \eqref{equ:def:2:10} with $V$ finite and each $K_q$ in ${}_{ \Bk }\!\prj$, whence
\begin{align*}
  \Hom_{ \Bk }( P,\Bk ) & = \Hom_{ \Bk }\!\big( \coprod_{ q \in V } Q( q,- ) \Tensor{\Bk} K_q,\Bk \big) && \\
  & \cong \coprod_{ q \in V } \Hom_{ \Bk }\!\big( Q( q,- ) \Tensor{\Bk} K_q,\Bk \big) && \mbox{$V$ finite} \\
  & \cong \coprod_{ q \in V } Q( -,\SerreFunctor q ) \Tensor{\Bk} \Hom_{ \Bk }( K_q,\Bk ) && \mbox{Lemma \ref{lem:16A_28_35_56}(i)}
\end{align*}
is in $\prj_Q$.
\end{proof}

\begin{Lemma}
\label{lem:17a_17b_17c:1}
Let $p \in Q_0$ be given and let $P$ be a chain complex over ${}_{ Q }\!\prj$.
\begin{enumerate}
\setlength\itemsep{4pt}

  \item  Each of $P( p )$, $S\{ p \} \Tensor{Q} P$, and $\Hom_{ \Bk }( S\{ p \} \Tensor{Q} P,\Bk )$ is a chain complex over ${}_{ \Bk }\!\prj$ and a semiprojective chain complex over ${}_{ \Bk }\!\Mod$.

  \item  If $P$ is acyclic then $P( p )$ is a contractible (also known as split exact) chain complex over ${}_{ \Bk }\!\Mod$.

\end{enumerate}
\end{Lemma}

\begin{proof}
(i): $P$ is a chain complex of objects of the form \eqref{equ:def:2:10} with $V$ finite and each $K_q$ in ${}_{ \Bk }\!\prj$.

It follows that $P( p )$ is a chain complex of modules of the form
\[
  \coprod_{ q \in V } Q( q,p ) \Tensor{\Bk} K_q
\]
with $V$ finite and each $K_q$ in ${}_{ \Bk }\!\prj$.  They are in ${}_{ \Bk }\!\prj$ because so is each $Q( q,p )$ by condition (Hom finiteness) in Section \ref{subsec:notation}.  So $P( p )$ is a chain complex over ${}_{ \Bk }\!\prj$.

It also follows that $S\{ p \} \Tensor{Q} P$ is a chain complex of modules of the form
\[
  S\{ p \} \Tensor{Q} \coprod_{ q \in V } Q( q,- ) \Tensor{\Bk} K_q
  \cong
  \coprod_{ q \in V } S\{ p \} \Tensor{Q} Q( q,- ) \Tensor{\Bk} K_q
  \cong
  \coprod_{ q \in V } S\{ p \}( q ) \Tensor{\Bk} K_q
\]
with $V$ finite and each $K_q$ in ${}_{ \Bk }\!\prj$, where we used \cite[p.\ 93]{Oberst-Roehrl}.  They are in ${}_{ \Bk }\!\prj$ because so is each $S\{ p \}( q )$.  So $S\{ p \} \Tensor{Q} P$ is a chain complex over ${}_{ \Bk }\!\prj$, hence so is $\Hom_{ \Bk }( S\{ p \} \Tensor{Q} P,\Bk )$.  

To finish the proof of (i), observe that $P( p )$, $S\{ p \} \Tensor{Q} P$, and $\Hom_{ \Bk }( S\{ p \} \Tensor{Q} P,\Bk )$ are, in particular, chain complexes over ${}_{ \Bk }\!\Prj$, hence semiprojective chain complexes over ${}_{ \Bk }\!\Mod$ because $\Bk$ is hereditary; see \cite[prop.\ 3.4(P)]{Avramov-Foxby_Unbounded}.

(ii): Suppose that $P$ is acyclic.  Then so is $P( p )$, so the morphism $P( p ) \xrightarrow{} 0$ to the zero chain complex is a quasi-isomorphism.  But $P( p )$ is semiprojective by part (i) so $P( p ) \xrightarrow{} 0$ is a homotopy equivalence by \cite[cor.\ 5.2.21]{CFH-book}.  Hence $P( p )$ is contractible.
\end{proof}

\begin{Lemma}
\label{lem:pre:3A}
Let $p \in Q_0$ be given, let $P$ be a chain complex  over ${}_{ Q }\!\prj$, and let $\partial_j$ be the $j$\!'th differential of $P$.  Then $S \{ p \} \Tensor{Q} \partial_j = 0$ if and only if $\Hom_Q( \partial_j,S \langle p \rangle ) = 0$.
\end{Lemma}

\begin{proof}
Lemma \ref{lem:17a_17b_17c:1}(i) says that $S\{ p \} \Tensor{Q} P$ is a chain complex over ${}_{ \Bk }\!\prj$ whence $S\{ p \} \Tensor{Q} \partial_j = 0$ if and only if $\Hom_{ \Bk }( S\{ p \} \Tensor{Q} \partial_j,\Bk ) = 0$.  The following isomorphisms now complete the proof:
\begin{align*}
  \Hom_{ \Bk }( S\{ p \} \Tensor{Q} \partial_j,\Bk )
  & \cong
  \Hom_Q\!\big( \partial_j,\Hom_{ \Bk }( S\{ p \},\Bk ) \big) && \mbox{Adjointness} \\
  & \cong
  \Hom_Q( \partial_j,S\langle p \rangle ). && \mbox{Equation \eqref{equ:S_formulae_1}}
\qedhere
\end{align*}
\end{proof}

\begin{Lemma}
\label{lem:17a_17b_17c:2}
Let $p \in Q_0$ be given and let $P$ be a chain complex  over ${}_{ Q }\!\prj$.
\begin{enumerate}
\setlength\itemsep{4pt}
  
  \item  If $L \in {}_{ A }\!\Prj$ is given, then $\Hom_Q( P,S_pL )$ is a semiprojective chain complex over ${}_{ A }\!\Mod$.
  
  \item  If $J \in {}_{ A }\!\Inj$ is given, then $\Hom_Q( P,S_pJ )$ is a semiinjective chain complex over ${}_{ A }\!\Mod$.

\end{enumerate}
\end{Lemma}

\begin{proof}
(i):  We can compute as follows:
\begin{align*}
  \Hom_Q( P,S_pL ) 
  & \cong \Hom_Q\!\big( P,\Hom_{ \Bk }( S\{ p \},L ) \big) && \mbox{Equation \eqref{equ:S_formulae_2}} \\
  & \cong \Hom_{ \Bk }( S\{ p \} \Tensor{Q} P,L ) && \mbox{Adjointness} \\
  & \cong \Hom_{ \Bk }( S\{ p \} \Tensor{Q} P,\Bk \Tensor{\Bk} L ) && \\
  & \cong \Hom_{ \Bk }( S\{ p \} \Tensor{Q} P,\Bk ) \Tensor{\Bk} L && \mbox{\cite[thm.\ 4.5.10(3+a)]{CFH-book} plus Lem.\ \ref{lem:17a_17b_17c:1}(i)} \\
  & = (*).
\end{align*}
Since $\Hom_{ \Bk }( S\{ p \} \Tensor{Q} P,\Bk )$ is a semiprojective chain complex over ${}_{ \Bk }\!\Mod$ by Lemma \ref{lem:17a_17b_17c:1}(i) while $L$ is in ${}_{ A }\!\Prj$, hence a semiprojective chain complex over ${}_{ A }\!\Mod$ if viewed as concentrated in degree zero, it follows from \cite[prop.\ 5.2.22]{CFH-book} that $(*)$ is a semiprojective chain complex over ${}_{ A }\!\Mod$.

(ii):  Computing as in the proof of part (i), we have:
\[
  \Hom_Q( P,S_pJ ) 
  \cong \Hom_Q\!\big( P,\Hom_{ \Bk }( S\{ p \},J ) \big) 
  \cong \Hom_{ \Bk }( S\{ p \} \Tensor{Q} P,J )
  = (**).
\]
Since $S\{ p \} \Tensor{Q} P$ is a semiprojective chain complex over ${}_{ \Bk }\!\Mod$ by Lemma \ref{lem:17a_17b_17c:1}(i) while $J$ is in ${}_{ A }\!\Inj$, hence a semiinjective chain complex over ${}_{ A }\!\Mod$ if viewed as concentrated in degree zero, it follows from \cite[prop.\ 5.3.25]{CFH-book} that $(**)$ is a semiinjective chain complex over ${}_{ A }\!\Mod$.
\end{proof}

\begin{Lemma}
\label{lem:40}
Let $p \in Q_0$ and $j \in \BZ$ be given.  For $X \in {}_{ Q,A }\!\Mod$ and $M \in {}_{ A }\!\Mod$, there are isomorphisms
\begin{align*}
  \Ext_A^j( E_pX,M ) & \cong \Ext_{ Q,A }^j( X,G_pM ), \\[1mm]
  \Ext_A^j( M,E_pX ) & \cong \Ext_{ Q,A }^j( F_pM,X )
\end{align*}
in ${}_{ \Bk }\!\Mod$.  
\end{Lemma}

\begin{proof}
The functor $E_p$ is exact because it acts by evaluating at $p$.  The functors $F_p$ and $G_p$ are exact by \cite[cor.\ 3.9]{HJ-JLMS} combined with condition (Hom-finiteness) in Section \ref{subsec:notation}.  The lemma hence follows from \cite[lem.\ 5.1]{HJ-Kyoto}.
\end{proof}

\begin{Lemma}
\label{lem:43}
The following conditions are equivalent for $X$ in ${}_{ Q,A }\!\Mod$:
\begin{enumerate}
\setlength\itemsep{4pt}
 
  \item  $X \in {}_{ Q,A }\cE$.

  \item  If $p \in Q_0$ then $\Ext^1_{ Q,A }( S_pA,X ) = 0$.

  \item  If $p \in Q_0$ and $J \in {}_{ A }\!\Inj$ then $\Ext^1_{ Q,A }( X,S_pJ ) = 0$.

\end{enumerate}
\end{Lemma}

\begin{proof}
Conditions (i) and (ii) are equivalent by the proof of \cite[thm.\ D]{HJ-TAMS}, and (i) implies (iii) by the first inclusion of \cite[prop.\ 3.2(b)]{HJ-TAMS}.

Suppose that (iii) holds.  Let $p \in Q_0$ and $K \in {}_{ \Bk }\!\Inj$ be given and set $J = \Hom_{ \Bk }( A,K )$.  We have $J \in {}_{ A }\!\Inj$ by \cite[cor.\ 5.4.28(a)]{CFH-book}, and inserting $J$ into (iii) gives
\[
  \Ext^1_{ Q,A }\!\big( X,S_p\Hom_{ \Bk }( A,K ) \big) = 0.
\]
But $S_p\Hom_{ \Bk }( A,K )$ and $\Hom_{ \Bk }( A,S_pK )$ are isomorphic because both are $Q$-shaped diagrams with $\Hom_{ \Bk }( A,K )$ placed at $p$ and $0$ elsewhere.  So the previous formula implies
\[
  \Ext^1_{ Q,A }\!\big( X,\Hom_{ \Bk }( A,S_pK ) \big) = 0,
\]
which by \cite[lem.\ 4.3(b)]{HJ-JLMS} reads
\[
  \Ext^1_Q( X,S_pK ) = 0.
\]
Since this holds for each $p \in Q_0$ and $K \in {}_{ \Bk }\!\Inj$, condition (i) holds by \cite[thm.\ 7.25]{HJ-JLMS} combined with the observation ${}_{ \Bk }\Inj = {}_{ \Bk }\GInj$, which is true by \cite[thm.\ 2.22]{Holm-GHD} because $\Bk$ is hereditary.
\end{proof}

\section{Complete projective resolutions}
\label{sec:complete}

This section constructs a complete projective resolution $P\langle q \rangle$ which can be used to define the functor $\LFunctorHomAngles{q}$ in Section \ref{sec:L}.

The following definition goes back to \cite[p.\ 67]{Mangeney-Peskine-Szpiro} and was formalised in \cite[def.\ 5.1]{Enochs-Jenda-GorInjProj}.

\begin{Definition}
\label{def:Gorenstein_projective_objects}
Let $\cA$ be an abelian category with enough projective objects.

A Gorenstein projective object $G$ of $\cA$ is an object which has a complete projective resolution $P$, that is, a chain complex $P$ permitting a diagram
\[
\vcenter{
\xymatrix @-5.5pc @! {
  P = \cdots \ar[rr] && P_1 \ar^{ \partial_1 }[rr] && P_0 \ar^{ \partial_0 }[rr] \ar@{->>}_<<<{ \varepsilon }[dr] && P_{ -1 } \ar^{ \partial_{ -1 } }[rr] && P_{ -2 } \ar[rr] && \cdots, \\
  &&&&& G = \Coker \partial_1 
                              \ar_>>>{ \eta }[ur] &&&& \\
                  }
        }
\]
where $\varepsilon$ and $\eta$ are the canonical morphisms, such that the following conditions are satisfied:
\begin{itemize}
\setlength\itemsep{4pt}

  \item  Each $P_i$ is a projective object of $\cA\!$.

  \item  $P$ is acyclic.
  
  \item  If $L$ is a projective object of $\cA\!$, then $\cA( P,L )$ is acyclic.

\end{itemize}
The three bullets can be summed up by saying that $P$ is a totally acyclic complex of projective objects.
\end{Definition}

\begin{Example}
Each projective object $L$ is a Gorenstein projective object because we can let the diagram in Definition \ref{def:Gorenstein_projective_objects} be given as follows:
\[
\vcenter{
\xymatrix @-1.55pc @! {
  \cdots \ar[rr] && 0 \ar[rr] && L \ar^{ \id }[rr] \ar@{->>}_<<<{ \id }[dr] && L \ar[rr] && 0 \ar[rr] && \cdots. \\
  &&&&& 
        L \ar_>>>>{ \id }[ur] &&&& \\
                  }
        }
\]
\end{Example}

\begin{Example}
If $q \in Q_0$ then \cite[cor.\ 4.8]{DSS} implies that $S\langle q \rangle$ is a Gorenstein projective object of ${}_{ Q }\!\Mod$, and dually, $S\{ q \}$ is a Gorenstein projective object of $\Mod_Q$.  Note that except for (very) special cases, $S\langle q \rangle$ and $S\{ q \}$ are not projective objects.
\end{Example}

Before constructing a complete projective resolution, we state two definitions expressing a weak form of minimality.

\begin{Definition}
\label{def:P_angles}
A complete projective resolution $P\langle q \rangle$ of $S\langle q \rangle$ is called good if it satisfies the following:
\begin{enumerate}
\setlength\itemsep{4pt}

  \item  The objects in degrees $0$ and $-1$ satisfy
$P\langle q \rangle_0 \cong Q( q,- )$ and $P\langle q \rangle_{ -1 } \cong Q( \SerreFunctor^{ -1 }q,- )$.
  
  \item  The zeroth differential $\partial_0$ of $P\langle q \rangle$ satisfies $\Hom_Q( S\langle p \rangle,\partial_0 ) = 0$ and $\Hom_Q( \partial_0,S\langle p \rangle ) = 0$ for each $p \in Q_0$.

\end{enumerate}
\end{Definition}

Definition \ref{def:P_angles} can be applied to $Q^{ \opp }$.  Then it concerns $S\{ q \}$, the contravariant counterpart of $S\langle q \rangle$, and has the following form, where we recall from Section \ref{subsec:notation} that the Serre functor of $Q^{ \opp }$ is $\SerreFunctor^{ -1 }$ viewed as a functor $Q^{ \opp } \xrightarrow{} Q^{ \opp }$.

\begin{Definition}
\label{def:P_brackets}
A complete projective resolution $P\{ q \}$ of $S\{ q \}$ is called good if it satisfies the following:
\begin{enumerate}
\setlength\itemsep{4pt}

  \item  The objects in degrees $0$ and $-1$ satisfy $P\{ q \}_0 \cong Q( -,q )$ and $P\{ q \}_{ -1 } \cong Q( -,\SerreFunctor q )$.
 
  \item  The zeroth differential $\partial_0$ of $P\{ q \}$ satisfies $\Hom_{ Q^{ \opp } }( S\{ p \},\partial_0 ) = 0$ and $\Hom_{ Q^{ \opp } }( \partial_0,S\{ p \} ) = 0$ for each $p \in Q_0$.

\end{enumerate}
\end{Definition}

We now recall from \cite[sec.\ 2.4.3]{Jasso-Q-shaped} and \cite[constr.\ 4.20]{Slaftsos-Vitoria} the construction of a complete projective resolution.

\begin{Definition}
\label{def:3}
Let $q \in Q_0$ be given.  Starting from $S\langle q \rangle$ we construct a chain complex $P\langle q \rangle$ over ${}_{ Q }\!\Mod$ by gluing the short exact sequences of the following diagram:
\[
\vcenter{
\xymatrix @-5.5pc @! {
  &&& \Omega^1 S\langle q \rangle \ar@{}[dr]^(.25){}="a"^(.80){}="b" \ar@{^(->}^{ \zeta_1 } "a";"b" &&&& \Omega^{ -1 } S\langle q \rangle \ar@{}[dr]^(.25){}="a"^(.80){}="b" \ar@{^(->}^{ \eta_{ -1 } } "a";"b" \\
  \cdots \ar[rr] && P\langle q \rangle_1 \ar^{ \partial_1 }[rr] \ar@{->>}^{ \varepsilon_1 }[ur] && P\langle q \rangle_0 \ar^{ \partial_0 }[rr] \ar@{->>}_<<<{ \varepsilon_0 }[dr] && P\langle q \rangle_{ -1 } \ar^{ \partial_{ -1 } }[rr] \ar@{->>}^<<<<{ \theta_{ -1 } }[ur] && P\langle q \rangle_{ -2 } \ar[rr] \ar@{->>}^>>>{ \theta_{ -2 } }[dr] && \cdots. \\
  & \Omega^2 S\langle q \rangle \ar@{^(->}^<<<{ \zeta_2 }[ur] &&&& \Omega^0 S\langle q \rangle = S\langle q \rangle \ar@{^(->}_>>>{ \eta_0 }[ur] &&&& \Omega^{ -2 } S\langle q \rangle \\
                  }
        }
\]
The sequences are obtained iteratively using the short exact sequences
\begin{gather*}
  0 \xrightarrow{} \BK X \xrightarrow{ \zeta } \BF X \xrightarrow{ \varepsilon } X \xrightarrow{} 0, \\[1mm]
  0 \xrightarrow{} Y \xrightarrow{ \eta } \BG Y \xrightarrow{ \theta } \BC Y \xrightarrow{} 0
\end{gather*}
of \cite[prop.\ 5.7]{HJ-TAMS} in the special case $A = \Bk$, where $X$ runs through $\Omega^i S\langle q \rangle$ for $i \geqslant 0$ and $Y$ runs through $\Omega^i S\langle q \rangle$ for $i \leqslant 0$.  The functors $\BK$, $\BF$, $\BG$, $\BC$ were defined in \cite[constr.\ 5.6]{HJ-TAMS}.
\end{Definition}

\begin{Lemma}
\label{lem:3}
Let $q \in Q_0$ be given and consider $P\langle q \rangle$ of Definition \ref{def:3}.
\begin{enumerate}
\setlength\itemsep{4pt}

  \item  $P\langle q \rangle$ is a complete projective resolution of $S\langle q \rangle$ consisting of objects of ${}_{ Q }\!\prj$.
  
  \item  $P\langle q \rangle$ satisfies Definition \ref{def:P_angles}(i).

  \item  If $\SerreFunctor q \neq q$ then $P\langle q \rangle$ also satisfies Definition \ref{def:P_angles}(ii) whence $P\langle q \rangle$ is good.

\end{enumerate}
The condition $\SerreFunctor q \neq q$ is satisfied in many examples, some of which are stated in the introduction after Theorem \ref{thm:B}.
\end{Lemma}

\begin{proof}
(i):  First, we show that $P\langle q \rangle$ satisfies the three bullets in Definition \ref{def:Gorenstein_projective_objects}:

It follows from \cite[lem.\ 5.8(b)]{HJ-TAMS} that $P\langle q \rangle$ consists of objects of ${}_{ Q }\!\prj$, hence of projective objects of ${}_{ Q }\!\Mod$.  

By construction, $P\langle q \rangle$ is acyclic.  

For the claim that $\Hom_Q( P\langle q \rangle,L )$ is acyclic when $L$ is projective, observe that each $\Omega^iS\langle q \rangle$ is Gorenstein projective in ${}_{ Q }\!\Mod$ by \cite[cor.\ 4.8]{DSS} because $( \Omega^iS\langle q \rangle )( p ) \in {}_{ \Bk }\!\prj$ for each $p \in Q_0$; the last fact follows from \cite[lem.\ 5.8(b)]{HJ-TAMS}.  Hence, up to degree shift, the truncated chain complex $\cdots \xrightarrow{} P\langle q \rangle_{ i+1 } \xrightarrow{} P\langle q \rangle_i$ is a projective resolution of a Gorenstein projective object for each $i$.  This implies that the chain complex $\Hom_Q( P\langle q \rangle_i,L ) \xrightarrow{} \Hom_Q( P\langle q \rangle_{ i+1 },L ) \xrightarrow{} \cdots$ only has non-zero homology at the start by \cite[prop.\ 2.3]{Holm-GHD}.  It follows that $\Hom_Q( P\langle q \rangle,L )$ is acyclic.

Secondly, $P\langle q \rangle$ permits the diagram in Definition \ref{def:3}, which can be used in Definition \ref{def:Gorenstein_projective_objects}, so $P\langle q \rangle$ is indeed a complete projective resolution of $S\langle q \rangle$.

In the proofs of parts (ii) and (iii) of the lemma, note that the functors $E_q$, $F_q$, $G_q$ of Equation \eqref{equ:EFG} are considered in the special case $A = \Bk$. 

(ii):  We have
\begin{align*}
  P\langle q \rangle_0 & = \BF S\langle q \rangle && \mbox{Definition \ref{def:3}} \\
  & = \coprod_{ p \in Q_0 } F_pE_p( S\langle q \rangle ) && \mbox{\cite[constr.\ 5.6]{HJ-TAMS}} \\
  & = \coprod_{ p \in Q_0 } F_p\big( S\langle q \rangle( p ) \big) && \mbox{Equation \eqref{equ:EFG}} \\
  & \cong F_q( \Bk ) && \mbox{Equation \eqref{equ:S_formulae_0}} \\
  & = Q( q,- ) \Tensor{\Bk} \Bk && \mbox{Equation \eqref{equ:EFG}} \\
  & \cong Q( q,- ) && 
\end{align*}
and
\begin{align*}
  P\langle q \rangle_{ -1 } & = \BG S\langle q \rangle && \mbox{Definition \ref{def:3}} \\
  & = \prod_{ p \in Q_0 } G_pE_p( S\langle q \rangle ) && \mbox{\cite[constr.\ 5.6]{HJ-TAMS}} \\
  & = \prod_{ p \in Q_0 } G_p\big( S\langle q \rangle( p ) \big) && \mbox{Equation \eqref{equ:EFG}} \\
  & \cong G_q( \Bk ) && \mbox{Equation \eqref{equ:S_formulae_0}} \\
  & \cong F_{\SerreFunctor^{-1}q}( \Bk ) && \mbox{\cite[lem.\ 3.4]{HJ-TAMS}} \\
  & = Q( \SerreFunctor^{-1}q,- ) \Tensor{\Bk} \Bk && \mbox{Equation \eqref{equ:EFG}} \\
  & \cong Q( \SerreFunctor^{-1}q,- ). && 
\end{align*}

(iii):  Observe that we have
\begin{align*}
  \Hom_Q\!\big( S\langle p \rangle,Q( r,- ) \big) & \cong \Hom_Q( S\langle p \rangle,F_r\Bk ) && \mbox{Equation \eqref{equ:EFG}} \\
  & \cong \Hom_Q( S\langle p \rangle,G_{ \SerreFunctor r }\Bk ) && \mbox{\cite[lem. 3.4]{HJ-TAMS}} \\
  & \cong \Hom_{ \Bk }( E_{ \SerreFunctor r }S\langle p \rangle,\Bk ) && \mbox{Adjointness} \\
  & \cong \Hom_{ \Bk }\!\big( S\langle p \rangle( \SerreFunctor r ),\Bk \big) && \mbox{Equation \eqref{equ:EFG}} \\
  & \cong  
  \left\{
    \begin{array}{cl}
      \Bk & \mbox{for $p=\SerreFunctor r$,} \\
      0   & \mbox{otherwise.}
    \end{array}
  \right. && \mbox{Equation \eqref{equ:S_formulae_0}}
\end{align*}
This implies that if $\SerreFunctor q \neq q$, then for each $p \in Q_0$, at least one of $\Hom_Q\!\big( S\langle p \rangle,Q( q,- ) \big)$ and $\Hom_Q\!\big( S\langle p \rangle,Q( \SerreFunctor^{ -1 }q ,- ) \big)$ is zero.  Hence Definition \ref{def:P_angles}(i) implies $\Hom_Q( S\langle p \rangle,\partial_0 ) = 0$.  This proves the first equation in Definition \ref{def:P_angles}(ii), and the second equation follows by a similar argument.
\end{proof}

We end the section with Lemma \ref{lem:pseudo-minimality}, which provides additional vanishing in the good case, and Lemma \ref{lem:dualising_P_angles}, which links Definitions \ref{def:P_angles} and \ref{def:P_brackets} by exploiting their symmetries.

\begin{Lemma}
\label{lem:pseudo-minimality}
\begin{enumerate}
\setlength\itemsep{4pt}

  \item  If $P\langle q \rangle$ is a good, complete projective resolution of $S\langle q \rangle$, then the first differential $\partial_1$ of $P\langle q \rangle$ satisfies
\[
  \Hom_Q( \partial_1,S\langle p \rangle ) = 0
\]
for each $p \in Q_0$.  

  \item  If $P\{ q \}$ is a good, complete projective resolution of $S\{ q \}$, then the first differential $\partial_1$ of $P\{ q \}$ satisfies
\[  
  \Hom_{ Q^{ \opp } }( \partial_1,S\{ p \} ) = 0
\]  
for each $p \in Q_0$.  

\end{enumerate}
\end{Lemma}

\begin{proof}
By symmetry, it is enough to prove part (i).  Since $P\langle q \rangle$ is a complete projective resolution of $S\langle q \rangle$, there is an exact sequence
\[
\vcenter{
\xymatrix @-2.5pc @! {
  P\langle q \rangle_1 \ar^-{ \partial_1 }[rr] && P\langle q \rangle_0 \ar@{->>}^-{ \varepsilon }[rr] && S\langle q \rangle.
                  }
        }
\]
By Definition \ref{def:P_angles}(i) this reads
\[
\vcenter{
\xymatrix @-2.5pc @! {
  P\langle q \rangle_1 \ar^-{ \partial_1 }[rr] && Q( q,- ) \ar@{->>}^-{\varepsilon}[rr] && S\langle q \rangle.
                  }
        }
\]
To prove part (i) of the lemma, it is enough to prove that the following extension problem can always be solved:
\[
\vcenter{
  \xymatrix @+0.5pc {
    Q( q,- ) \ar@{->>}^-{ \varepsilon }[r] \ar_{ \varepsilon' }[d] & S\langle q \rangle \ar@{.>}[dl] \\
    S\langle p \rangle. \\
                    }
        }
\]
To prove this, it is enough to show
\begin{equation}
\label{equ:lem:pseudo-minimality:20} 
  \Ker \varepsilon \subseteq \Ker \varepsilon'.
\end{equation}

Suppose $p \neq q$.  Then \eqref{equ:lem:pseudo-minimality:20} holds because $\varepsilon' = 0$ since Yoneda's Lemma and Equation \eqref{equ:S_formulae_0} give
\[
  \Hom_Q\!\big( Q( q,- ),S\langle p \rangle \big) \cong  S\langle p \rangle( q ) = 0.
\]

Suppose $p = q$.  Since $\fr( q,- )$ annihilates $S\langle p \rangle = S\langle q \rangle = Q( q,- )/\fr( q,- )$, we have $\fr( q,- ) \subseteq \Ker \varepsilon'$, so to show \eqref{equ:lem:pseudo-minimality:20} it is enough to show
\begin{equation}
\label{equ:lem:pseudo-minimality:30}
  \Ker \varepsilon = \fr( q,- ).
\end{equation}
The inclusion $\supseteq$ in Equation \eqref{equ:lem:pseudo-minimality:30} is clear because $\fr( q,- )$ annihilates $S\langle q \rangle$.  

The inclusion $\subseteq$ in Equation \eqref{equ:lem:pseudo-minimality:30}: Observe that $\varepsilon$ is a natural transformation and that if $r \neq q$ then $\Ker \varepsilon_r \subseteq \fr( q,r )$ because $\fr( q,r ) = Q( q,r )$.  

To complete the proof, we must show that we also have $\Ker \varepsilon_q \subseteq \fr( q,q )$.  Consider an element $\kappa \in \Ker \varepsilon_q$, and use condition (Strong retraction) in Section \ref{subsec:notation} to write
\begin{equation}
\label{equ:lem:pseudo-minimality:111}
  \kappa = \alpha\id_q + \rho
\end{equation}
with $\alpha \in \Bk$ and $\rho \in \fr( q,q )$ whence
\begin{equation}
\label{equ:lem:pseudo-minimality:131}
  0
  = \varepsilon_q( \kappa )
  = \varepsilon_q( \alpha\id_q + \rho )
  = \alpha\varepsilon_q( \id_q ).
\end{equation}
We used that $\varepsilon_q( \rho ) = 0$ by the inclusion $\supseteq$ in Equation \eqref{equ:lem:pseudo-minimality:30}.  Since $\varepsilon_q$ is surjective, there exists an element $\varphi \in Q( q,q )$ such that $\varepsilon_q( \varphi ) = \id_q + \fr( q,q )$.  Use condition (Strong retraction) in Section \ref{subsec:notation} again to write $\varphi = \beta\id_q + \sigma$ with $\beta \in \Bk$ and $\sigma \in \fr( q,q )$ whence
\begin{equation}
\label{equ:lem:pseudo-minimality:141}
  \beta\varepsilon_q( \id_q )
  = \varepsilon_q( \beta\id_q )
  = \varepsilon_q( \beta\id_q+\sigma )
  = \varepsilon_q( \varphi )
  = \id_q + \fr( q,q ).
\end{equation}
We used that $\varepsilon_q( \sigma ) = 0$ by the inclusion $\supseteq$ in Equation \eqref{equ:lem:pseudo-minimality:30} again.  Now consider the element $\alpha\beta\varepsilon_q( \id_q )$.  Equation \eqref{equ:lem:pseudo-minimality:131} shows that it is equal to zero, and Equation \eqref{equ:lem:pseudo-minimality:141} shows that it is equal to $\alpha\big(\! \id_q + \fr( q,q ) \big)$.  This implies $\alpha = 0$ whence Equation \eqref{equ:lem:pseudo-minimality:111} shows $\kappa = \rho$ so we have $\kappa \in \fr( q,q )$ as desired.
\end{proof}

\begin{Lemma}
\label{lem:dualising_P_angles}
Let $P\langle q \rangle$ be a complete projective resolution of $S\langle q \rangle$ consisting of objects of ${}_{ Q }\!\prj$.  Define
\[
  P\{ q \} = \Sigma^{ -1 }\Hom_{ \Bk }( P\langle q \rangle,\Bk ).
\]  
\begin{enumerate}
\setlength\itemsep{4pt}

  \item  $P\{ q \}$ is a complete projective resolution of $S\{ q \}$ consisting of objects of $\prj_Q$.

  \item  If $P\langle q \rangle$ is good, then $P\{ q \}$ is good.

\end{enumerate}
\end{Lemma}

\begin{proof}
(i): By definition of $P\{ q \}$ we have
\begin{equation}
\label{equ:lem:dualising_P_angles:10}
  P\{ q \}_i = \Hom_{ \Bk }( P\langle q \rangle_{ -i-1 },\Bk )
\end{equation}
for each $i$.  Since $P\langle q \rangle$ consists of objects of ${}_{ Q }\!\prj$, Lemma \ref{lem:algebraic_duality}(iii) gives that $P\{ q \}$ consists of objects of $\prj_Q$.

We now show that $P\{ q \}$ satisfies the three bullets in Definition \ref{def:Gorenstein_projective_objects}:

Each $P\{ q \}_i$ is a projective object of $\Mod_Q$ since it is in $\prj_Q$, see Definition \ref{def:2}.

Lemma \ref{lem:17a_17b_17c:1}(ii) implies that $P\langle q \rangle( p )$ is a contractible chain complex over ${}_{ \Bk }\!\Mod$ for each $p \in Q_0$.  Hence $\Hom_{ \Bk }\!\big( P\langle q \rangle( p ),\Bk \big)$ is a contractible chain complex over ${}_{ \Bk }\!\Mod$, so in particular an acyclic chain complex.  But we have $\Hom_{ \Bk }\!\big( P\langle q \rangle( p ),\Bk \big) = \Hom_{ \Bk }( P\langle q \rangle,\Bk )( p )$, and since this is an acyclic chain complex over ${}_{ \Bk }\!\Mod$ for each $p \in Q_0$, it follows that $\Hom_{ \Bk }( P\langle q \rangle,\Bk )$ is acyclic whence so is $P\{ q \}$.

For $L \in \Prj_Q$ we can compute as follows:
\begin{align*}
  \Hom_{ Q^{ \opp } }\!\big( P\{ q \},L \big)
  & = \Hom_{ Q^{ \opp } }\!\big( \Sigma^{ -1 }\Hom_{ \Bk }( P\langle q \rangle,\Bk ) ,L \big) && \\
  & \cong \Sigma\Hom_{ Q^{ \opp } }\!\big( \Hom_{ \Bk }( P\langle q \rangle,\Bk ) ,L \big) && \\
  & \cong \Sigma( L \circ \SerreFunctor ) \Tensor{Q} P\langle q \rangle. && \mbox{Lemma \ref{lem:16A_28_35_56}(v)}
\end{align*}
This is acyclic since $P\langle q \rangle$ is acyclic while $L \circ \SerreFunctor$ is in $\Prj_Q$ because $L$ is in $\Prj_Q$ and $\SerreFunctor$ is an autoequivalence of $Q$.

Finally, we show that $P\{ q \}$ resolves $S\{ q \}$:

Since $P\langle q \rangle$ is a complete projective resolution of $S\langle q \rangle$, there is an exact sequence
\[
\vcenter{
\xymatrix @-1.5pc @! {
  P\langle q \rangle_1 \ar^{ \partial^{ P\langle q \rangle}_1 }[rr] && P\langle q \rangle_0 \ar@{->>}^{ \varepsilon }[rr] && S\langle q \rangle.
                  }
        }
\]
Applying the functor $\Hom_{ \Bk }( -,\Bk )$ gives an exact sequence
\[
\vcenter{
\xymatrix @-3.3pc @! {
\Hom_{ \Bk }( S\langle q \rangle,\Bk ) \ar@{^(->}^{ \eta }[rr] && \Hom_{ \Bk }( P\langle q \rangle_0,\Bk ) \ar^{ \Hom_{ \Bk }( \partial^{ P\langle q \rangle}_1,\Bk ) }[rr] && \Hom_{ \Bk }( P\langle q \rangle_1,\Bk ).
                  }
        }
\]
But using Equations \eqref{equ:S_formulae_1} and \eqref{equ:lem:dualising_P_angles:10} this sequence can also be written
\[
\vcenter{
\xymatrix @-2.3pc @! {
S\{ q \} \ar@{^(->}^{ \eta }[rr] && P\{ q \}_{ -1 }\ar^{ \partial^{ P\langle q \rangle}_{ -1 } }[rr] && P\{ q \}_{ -2 },
                  }
        }
\]
so $P\{ q \}$ does indeed resolve $S\{ q \}$.

(ii): Suppose that $P\langle q \rangle$ satisfies Definition \ref{def:P_angles}(i).  Then Equation \eqref{equ:lem:dualising_P_angles:10} and Serre duality imply that $P\{ q \}$ satisfies Definition \ref{def:P_brackets}(i) as follows:
\begin{align*}
  P\{ q \}_0
  & = \Hom_{ \Bk }( P\langle q \rangle_{ -1 },\Bk ) 
  \cong \Hom_{ \Bk }\!\big( Q( \SerreFunctor^{ -1 }q,- ),\Bk \big)
  \cong Q( -,\SerreFunctor\SerreFunctor^{ -1 }q )
  \cong Q( -,q ), \\
  P\{ q \}_{ -1 }
  & = \Hom_{ \Bk }( P\langle q \rangle_0,\Bk ) 
  \cong \Hom_{ \Bk }\!\big( Q( q,- ),\Bk \big)
  \cong Q( -,\SerreFunctor q ).
\end{align*}

Suppose that $P\langle q \rangle$ satisfies Definition \ref{def:P_angles}(ii).  To prove that $P\{ q \}$ satisfies Definition \ref{def:P_brackets}(ii), we must show that if $p \in Q_0$ then each composition
\begin{equation}
\label{equ:lem:dualising_P_angles:100}
  S\{ p \} \xrightarrow{ f } P\{ q \}_0 \xrightarrow{ \partial^{ P\{ q \} }_0 } P\{ q \}_{ -1 }
\end{equation}
is zero and each composition
\begin{equation}
\label{equ:lem:dualising_P_angles:110}
  P\{ q \}_0 \xrightarrow{ \partial^{ P\{ q \} }_0 } P\{ q \}_{ -1 } \xrightarrow{ g } S\{ p \} 
\end{equation}
is zero.  Equation \eqref{equ:lem:dualising_P_angles:10} implies that \eqref{equ:lem:dualising_P_angles:100} is
\[
  S\{ p \} \xrightarrow{ f } \Hom_{ \Bk }( P\langle q \rangle_{ -1 },\Bk ) \xrightarrow{ \Hom_{ \Bk }( \partial^{ P\langle q \rangle }_0,\Bk ) } \Hom_{ \Bk }( P\langle q \rangle_0,\Bk ).
\]
Applying $\Hom_{ \Bk }( -,\Bk )$ to this gives
\[
  P\langle q \rangle_0 \xrightarrow{ \partial^{ P\langle q \rangle }_0 } P\langle q \rangle_{ -1 } \xrightarrow{ \Hom_{ \Bk }( f,\Bk ) } \Hom_{ \Bk }( S\{ p \},\Bk )
\]
up to isomorphism by Lemma \ref{lem:algebraic_duality}(i), which also implies that it is enough to show that this last composition is zero.  But that holds by Definition \ref{def:P_angles}(ii) because we have $\Hom_{ \Bk }( S\{ p \},\Bk ) \cong S\langle p \rangle$ by Equation \eqref{equ:S_formulae_1}.  Equation \eqref{equ:lem:dualising_P_angles:110} is handled similarly.
\end{proof}

\section{The functors $\LFunctorHomAngles{q}$}
\label{sec:L}

This section defines the functors $\LFunctorHomAngles{q}$ and establishes some of their properties.

\begin{Setup}
\label{set:P}
For each $q \in Q_0$ we pick, and keep in place for the rest of the paper, a chain complex $P\langle q \rangle$ which is a complete projective resolution of $S\langle q \rangle$ consisting of objects of ${}_{ Q }\!\prj$.  For instance, we could let $P\langle q \rangle$ be the resolution from Definition \ref{def:3}; see Lemma \ref{lem:3}(i).

We set
\[
  P\{ q \} = \Sigma^{ -1 }\Hom_{ \Bk }( P\langle q \rangle,\Bk ),
\]
and observe that $P\{ q \}$ is a complete projective resolution of $S\{ q \}$ consisting of objects of $\prj_Q$ by Lemma \ref{lem:dualising_P_angles}(i).
\end{Setup}

\begin{Remark}
\label{rmk:P}
We do not require $P\langle q \rangle$ to be the resolution from Definition \ref{def:3}, but if it were and if $q \in Q_0$ satisfied $\SerreFunctor q \neq q$, then $P\langle q \rangle$ would be good by Lemma \ref{lem:3}(iii), and $P\{ q \}$ would be good by Lemma \ref{lem:dualising_P_angles}(ii).
\end{Remark}

The following definition is due to \cite[sec.\ 2.4.3]{Jasso-Q-shaped}.

\begin{Definition}
\label{def:L}
Let $q \in Q_0$ be given.  We define a functor $\LFunctorHomAngles{q} : {}_{ Q }\!\Mod \xrightarrow{} \Ch( {}_{ \Bk }\!\Mod )$ by
\begin{equation}
\label{equ:def:L:10}
  \LFunctorHomAngles{q}( - ) = \Hom_Q( P\langle q \rangle,- )
\end{equation}
and a functor $\LFunctorHomBrackets{q} : \Mod_Q \xrightarrow{} \Ch( {}_{ \Bk }\!\Mod )$ by
\[
  \LFunctorHomBrackets{q}( - ) = \Hom_{ Q^{ \opp } }( P\{ q \},- ).
\]

The definitions are compatible with additional structures.  For instance, we will mainly view $\LFunctorHomAngles{q}$ and $\LFunctorHomBrackets{q}$ as the following functors:
\begin{gather*}
  \LFunctorHomAngles{q} : {}_{ Q,A }\!\Mod \xrightarrow{} \Ch( {}_{ A }\!\Mod ), \\
  \LFunctorHomBrackets{q} : \Mod_{ Q,A } \xrightarrow{} \Ch( \Mod_A ).
\end{gather*}
\end{Definition}

\begin{Remark}
\label{rmk:4}
For $q \in Q_0$ and each integer $j \geqslant 1$ we have
\[
  \H^j\LFunctorHomAngles{q}( - ) 
  \cong \Ext_Q^j( S\langle q \rangle, - )
  = \BH_{ [q] }^j( - ),
\]
where $\BH_{ [q] }^j$ is the generalised cohomology functor on ${}_{ Q,A }\!\Mod$ introduced in \cite[def.\ 7.11]{HJ-JLMS}.  The isomorphism holds because the left hand part of $P\langle q \rangle$ is a projective resolution of $S\langle q \rangle$, and the equality holds by definition.

On the other hand, if $j \leqslant 0$ then $\H^j\LFunctorHomAngles{q}( - )$ has no counterpart in \cite{HJ-JLMS}.
\end{Remark}

\begin{Lemma}
\label{lem:5i_5ii_30_32}
Let $q \in Q_0$ be given.
\begin{enumerate}
\setlength\itemsep{4pt}

  \item  The functor $\LFunctorHomAngles{q}$ is exact and respects set indexed products and coproducts.

  \item  Let $\cA \subseteq {}_{ A }\!\Mod$ be an additive subcategory and let $X \in {}_{ Q,A }\!\Mod$ be given.  Then the following implication holds:
\[
  X( p ) \in \cA \mbox{ for each } p \in Q_0 \Rightarrow \LFunctorHomAngles{q}X \mbox{ is a chain complex over $\cA\!$.}
\]

\end{enumerate}
\end{Lemma}

\begin{proof}
The functor $\LFunctorHomAngles{q}$ is exact because $P \langle q \rangle$ is a chain complex over ${}_{ Q }\!\Prj$.  Moreover, $P\langle q \rangle$ consists of objects in ${}_{ Q }\!\prj$ by assumption, so Lemma \ref{lem:pre_5i_5ii_30_32}(i) implies that $\LFunctorHomAngles{q}$ respects set indexed products and coproducts.  This proves part (i) of the lemma, and part (ii) follows from Lemma \ref{lem:pre_5i_5ii_30_32}(ii).
\end{proof}

\begin{Lemma}
\label{lem:17A_29_54_57_38A_40A}
Let $q \in Q_0$ be given.  

There are the following isomorphisms in $\Ch( {}_{ \Bk }\!\Mod )$, natural with respect to $M \in {}_{ \Bk }\!\Mod$ and $X \in {}_{ Q }\!\Mod$:
\begin{enumerate}
\setlength\itemsep{4pt}

  \item  $M \Tensor{\Bk} \LFunctorHomAngles{q} X \cong \LFunctorHomAngles{q}( M \Tensor{\Bk} X )$.

  \item  $\Hom_{ \Bk }( M,\LFunctorHomAngles{q}X ) \cong \LFunctorHomAngles{q}\Hom_{ \Bk }( M,X )$.

  \item  $\Hom_{ \Bk }( \LFunctorHomAngles{q}X,M ) \cong
  \Sigma^{ -1 }\LFunctorHomBrackets{q}\big( \Hom_{ \Bk }( X,M ) \circ \SerreFunctor^{ -1 } \big)$.

\end{enumerate}
There is the following isomorphism in $\Ch( {}_{ \Bk }\!\Mod )$, natural with respect to $N \in {}_{ \Bk }\!\Mod$ and $Y \in \Mod_Q$:
\begin{enumerate}
\setlength\itemsep{4pt}
\setcounter{enumi}{3}

  \item  $\Hom_{ \Bk }( \LFunctorHomBrackets{q}Y,N ) \cong \Sigma^{ -1 }\LFunctorHomAngles{q}\big( \Hom_{ \Bk }( Y,N ) \circ \SerreFunctor \big)$.

\end{enumerate}
\end{Lemma}

\begin{proof}
(i): Consider the $\Bk$-linear functor $F: {}_{ \Bk }\!\Mod \xrightarrow{} {}_{ \Bk }\!\Mod$ given by $F( - ) = M \Tensor{\Bk} -$.  We can compute as follows:
\[
  M \Tensor{\Bk} \LFunctorHomAngles{q}X
  = F\big( \Hom_Q( P\langle q \rangle,X ) \big)
  \cong \Hom_Q( P\langle q \rangle,F \circ X )
  = \LFunctorHomAngles{q}( M \Tensor{\Bk} X ).
\]
The equalities hold by definition and the isomorphism by Lemma \ref{lem:16A_28_35_56}(ii).

(ii):  This is proved like part (i) but replacing $F$ with $F( - ) = \Hom_{ \Bk }( M,- )$.

(iii): We can compute as follows:
\begin{align*}
  \Hom_{ \Bk }( \LFunctorHomAngles{q}X,M ) 
  & = \Hom_{ \Bk }\!\big( \Hom_Q( P\langle q \rangle,X),M \big) && \mbox{Definition \ref{def:L}} \\
  & \cong \Hom_{ Q^{ \opp } }\!\big( \Hom_{ \Bk }( P\langle q \rangle,\Bk ),\Hom_{ \Bk }( X,M ) \circ \SerreFunctor^{ -1 } \big) && \mbox{Lemma \ref{lem:16A_28_35_56}(vi)} \\
  & = \Hom_{ Q^{ \opp } }\!\big( \Sigma P\{ q \},\Hom_{ \Bk }( X,M ) \circ \SerreFunctor^{ -1 } \big) && \mbox{Setup \ref{set:P}} \\
  & \cong \Sigma^{ -1 }\Hom_{ Q^{ \opp } }\!\big( P\{ q \},\Hom_{ \Bk }( X,M ) \circ \SerreFunctor^{ -1 } \big)  \\
  & = \Sigma^{ -1 }\LFunctorHomBrackets{q}\big( \Hom_{ \Bk }( X,M ) \circ \SerreFunctor^{ -1 } \big). && \mbox{Definition \ref{def:L}}
\end{align*}

(iv): Proved similarly to part (iii).
\end{proof}

\begin{Remark}
\label{rmk:17A_29_54_57_38A_40A}
Lemma \ref{lem:17A_29_54_57_38A_40A} is compatible with additional structures.  For instance, the isomorphism in part (i) also exists in the version
\[
  M \Tensor{A} \LFunctorHomAngles{q} X \cong \LFunctorHomAngles{q}( M \Tensor{A} X ),
\]
natural with respect to $M \in \Mod_A$ and $X \in {}_{ Q,A }\!\Mod$.
\end{Remark}

\begin{Lemma}
\label{lem:17d}
Let $p,q \in Q_0$ be given.  
\begin{enumerate}
\setlength\itemsep{4pt}

  \item  If $L \in {}_{ A }\!\Prj$ is given, then $\LFunctorHomAngles{q}S_pL$ is a semiprojective chain complex over ${}_{ A }\!\Mod$.
  
  \item  If $J \in {}_{ A }\!\Inj$ is given, then $\LFunctorHomAngles{q}S_pJ$ is a semiinjective chain complex over ${}_{ A }\!\Mod$.

\end{enumerate}
\end{Lemma}

\begin{proof}
We have $\LFunctorHomAngles{q}S_p( - ) = \Hom_Q\!\big( P\langle q \rangle,S_p( - ) \big)$, and $P\langle q \rangle$ is a chain complex over ${}_{ Q }\!\prj$ by Setup \ref{set:P}.  Hence the lemma follows from Lemma \ref{lem:17a_17b_17c:2}.
\end{proof}

\begin{Lemma}
\label{lem:9A_13A}
Let $q \in Q_0$ and $M \in {}_{ A }\!\Mod$ be given.  Suppose that $P\langle q \rangle$ is good in the sense of Definition \ref{def:P_angles}.  Then the chain complex $\LFunctorHomAngles{q}S_qM$ over ${}_{ A }\!\Mod$ has $M$, viewed as a chain complex concentrated in degree zero, as a direct summand.
\end{Lemma}

\begin{proof}
Combining Equations \eqref{equ:S_formulae_2} and \eqref{equ:def:L:10} gives that $\LFunctorHomAngles{q}S_qM$ has differentials
\[
  \partial_{ \LFunctorHomAngles{q}S_qM }^i
  =
  \Hom_Q\!\big( \partial^{ P\langle q \rangle }_{ i+1 },\Hom_{ \Bk }( S\{ q \},M ) \big)
  \cong
  \Hom_{ \Bk }( S\{ q \} \Tensor{Q} \partial^{ P \langle q \rangle }_{ i+1 },M ) = (*),
\]
where the isomorphism is by adjointness.  Definition \ref{def:P_angles}(ii) and Lemmas \ref{lem:pre:3A} and \ref{lem:pseudo-minimality}(i) imply $(*) = 0$ for $i=-1$ and $i=0$.  Hence $\LFunctorHomAngles{q}S_qM$ has $( \LFunctorHomAngles{q}S_qM )^0$, viewed as a cochain complex concentrated in degree zero, as a direct summand.  However, $( \LFunctorHomAngles{q}S_qM )^0$ can be computed as follows:
\begin{align*}
  ( \LFunctorHomAngles{q}S_qM )^0
  & = \Hom_Q\!\big( P\langle q \rangle_0,\Hom_{ \Bk }( S\{ q \},M ) \big) && \mbox{Equations \eqref{equ:S_formulae_2} and \eqref{equ:def:L:10}} \\
  & \cong \Hom_{ \Bk }( S\{ q \} \Tensor{Q} P \langle q \rangle_0,M ) && \mbox{Adjointness} \\
  & \cong \Hom_{ \Bk }( S\{ q \} \Tensor{Q} Q( q,- ),M ) && \mbox{Definition \ref{def:P_angles}(i)} \\
  & \cong \Hom_{ \Bk }( S\{ q \}( q ),M ) && \mbox{\cite[p.\ 93]{Oberst-Roehrl}} \\
  & = \Hom_{ \Bk }( \Bk,M ) && \mbox{The dual of Equation \eqref{equ:S_formulae_0}} \\
  & \cong M. && \qedhere
\end{align*}
\end{proof}

\begin{Lemma}
\label{lem:5iii_a_70}
Let $p,q \in Q_0$ and $M \in {}_{ A }\!\Mod$ be given.  Then $\LFunctorHomAngles{q}F_pM$ is a contractible (also known as split exact) chain complex over $\add M$.
\end{Lemma}

\begin{proof}
Denoting $\SerreFunctor p$ by $r$ gives $F_pM \cong G_rM$ by \cite[lem.\ 3.4]{HJ-TAMS} whence $\LFunctorHomAngles{q}F_pM \cong \LFunctorHomAngles{q}G_rM$.  We can compute as follows:
\begin{align*}
  \LFunctorHomAngles{q}G_rM
  & = \Hom_Q\!\Big( P\langle q \rangle,\Hom_{ \Bk }\!\big( Q( -,r ),M \big) \Big) && \mbox{Equations \eqref{equ:EFG} and \eqref{equ:def:L:10}} \\
  & \cong \Hom_{ \Bk }\!\big( Q( -,r ) \Tensor{Q} P\langle q \rangle,M \big) && \mbox{Adjointness} \\
  & \cong \Hom_{ \Bk }\!\big( P\langle q \rangle( r ),M \big). && \mbox{\cite[p.\ 93]{Oberst-Roehrl}}
\end{align*}
This implies the lemma because $P\langle q \rangle( r )$ is a contractible chain complex over ${}_{ \Bk }\!\prj$ by Lemma \ref{lem:17a_17b_17c:1}.
\end{proof}

\begin{Lemma}
\label{lem:5iii_b}
Let $q \in Q_0$ be given.  
\begin{enumerate}
\setlength\itemsep{4pt}

  \item  If $L \in {}_{ Q,A }\!\Prj$ is given, then $\LFunctorHomAngles{q}L$ is a projective object of $\Ch( {}_{ A }\!\Mod )$.  

  \item  If $J \in {}_{ Q,A }\!\Inj$ is given, then $\LFunctorHomAngles{q}J$ is an injective object of $\Ch( {}_{ A }\!\Mod )$.  

\end{enumerate}
\end{Lemma}

\begin{proof}
(ii): By \cite[proof of thm.\ 7.29]{HJ-JLMS} there exist modules $J_p \in {}_{ A }\!\Inj$ for $p \in Q_0$ such that $J \cong \prod_{ p \in Q_0 } G_pJ_p$.  Since $\LFunctorHomAngles{q}$ respects set indexed products by Lemma \ref{lem:5i_5ii_30_32}(i), it is hence enough to show that $\LFunctorHomAngles{q}G_pJ_p$ is an injective object of $\Ch( {}_{ A }\!\Mod )$ for each $p \in Q_0$.  And we have $\LFunctorHomAngles{q}G_pJ_p \cong \LFunctorHomAngles{q}F_{ \SerreFunctor^{ -1 }p }J_p$ by \cite[lem.\ 3.4]{HJ-TAMS}, so it is enough to show that $\LFunctorHomAngles{q}F_{ \SerreFunctor^{ -1 }p }J_p$ is an injective object of $\Ch( {}_{ A }\!\Mod )$ for each $p \in Q_0$.

But $\LFunctorHomAngles{q}F_{ \SerreFunctor^{ -1 }p }J_p$ is a contractible chain complex over $\add J_p$ by Lemma \ref{lem:5iii_a_70}; in particular, it is a contractible chain complex over ${}_{ A }\!\Inj$.  Combining this with \cite[prop.\ 3.2]{Gillespie-degreewise} applied to the injective cotorsion pair $( {}_{ A }\!\Mod,{}_{ A }\!\Inj )$ in ${}_{ A }\!\Mod$ shows that as desired, $\LFunctorHomAngles{q}F_{ \SerreFunctor^{ -1 }p }J_p$ is an injective object of $\Ch( {}_{ A }\!\Mod )$.

(i): Proved similarly to part (ii), but without the need to invoke \cite[lem.\ 3.4]{HJ-TAMS}.
%
\end{proof}

\begin{Lemma}
\label{lem:5v}
Let $X \in {}_{ Q }\!\Mod$ be given.  The following conditions are equivalent:
\begin{enumerate}
\setlength\itemsep{4pt}

  \item  $X \in {}_{ Q }\cE$,
  
  \item  For each $q \in Q_0$, the chain complex $\LFunctorHomAngles{q}X$ is acyclic.

\end{enumerate}
\end{Lemma}

\begin{proof}
(i)$\Rightarrow$(ii): Suppose that (i) holds, that is, $X$ has finite projective dimension in ${}_{ Q }\!\Mod$ by Equation \eqref{equ:exact_objects}.  Let $q \in Q_0$ be given.  Then $\LFunctorHomAngles{q}X = \Hom_Q( P\langle q \rangle,X )$ is acyclic by \cite[prop.\ 2.3]{Holm-GHD} because $P\langle q \rangle$ is a complete projective resolution.

(ii)$\Rightarrow$(i): Suppose that (ii) holds.  In particular, for each $q \in Q_0$ we have $\H^1\!\LFunctorHomAngles{q}X = 0$, and this implies $X \in {}_{ Q }\cE$ by Remark \ref{rmk:4} and \cite[thm.\ 7.1]{HJ-JLMS}.
\end{proof}

\section{The adjoint functors $\JFunctorHomAngles{q}$ and $\MFunctorHomAngles{q}$}
\label{sec:JM}

This section shows that the adjoint functors $\JFunctorHomAngles{q}$ and $\MFunctorHomAngles{q}$ exist and establishes some of their properties.

\begin{Proposition}
\label{pro:13_14_33_21A}
Let $q \in Q_0$ be given.  The functor $\LFunctorHomAngles{q}: {}_{ Q,A }\!\Mod \xrightarrow{} \Ch( {}_{ A }\!\Mod )$ has a left adjoint functor
\[
  \JFunctorHomAngles{q} : \Ch( {}_{ A }\!\Mod ) \xrightarrow{} {}_{ Q,A }\!\Mod
\]
and a right adjoint functor
\[
  \MFunctorHomAngles{q} : \Ch( {}_{ A }\!\Mod ) \xrightarrow{} {}_{ Q,A }\!\Mod.
\]
\end{Proposition}

\begin{proof}
We will prove this using the version of the Special Adjoint Functor Theorem given in \cite[cor., p.\ 130]{MacLane-book}.

The left adjoint exists by \cite[cor., p.\ 130]{MacLane-book} as a consequence of the following properties:
\begin{itemize}
\setlength\itemsep{4pt}

  \item  ${}_{ Q,A }\!\Mod$ has small $\Hom$ sets and is well powered since it is defined as
\[  
  \{\, \mbox{$\Bk$-linear functors $Q \xrightarrow{} {}_{ A }\!\Mod$} \,\}
\]  
while $Q$ is small.
  
  \item  ${}_{ Q,A }\!\Mod$ has a small cogenerating set by \cite[prop.\ 3.12(b)]{HJ-JLMS} and set indexed limits by \cite[cor.\ X.4.4]{Stenstroem} since it is a Grothendieck category by \cite[prop.\ 3.12]{HJ-JLMS}.
  
  \item  $\Ch( {}_{ A }\!\Mod )$ has small $\Hom$ sets.
  
  \item  $\LFunctorHomAngles{q}$ respects set indexed limits since it is exact and respects set indexed products by Lemma \ref{lem:5i_5ii_30_32}(i).

\end{itemize}

For the right adjoint, consider the opposite functor $\LFunctorHomAngles{q}^{ \opp } : {}_{ Q,A }\!\Mod^{ \opp } \xrightarrow{} \Ch( {}_{ A }\!\Mod )^{ \opp }$ in the sense of \cite[p.\ 33]{MacLane-book}.  Then $\LFunctorHomAngles{q}$ has a right adjoint if $\LFunctorHomAngles{q}^{ \opp }$ has a left adjoint, see \cite[lem.\ 3.8]{HJ-Kyoto}, and this left adjoint exists by \cite[cor., p.\ 130]{MacLane-book} as a consequence of the following properties:
\begin{itemize}
\setlength\itemsep{4pt}

  \item  ${}_{ Q,A }\!\Mod^{ \opp }$ has small $\Hom$ sets because so does ${}_{ Q,A }\!\Mod$ by the previous part of the proof.  And ${}_{ Q,A }\!\Mod^{ \opp }$ is well powered because ${}_{ Q,A }\!\Mod$ is co-well powered.  To see that ${}_{ Q,A }\!\Mod$ is co-well powered, observe that it is well powered by the previous part of the proof, and that it is abelian whence quotients and subobjects correspond.

  \item  ${}_{ Q,A }\!\Mod^{ \opp }$ has a small cogenerating set and set indexed limits because ${}_{ Q,A }\!\Mod$ has a small generating set by \cite[prop.\ 3.12(a)]{HJ-JLMS} and set indexed colimits by definition since it is a Grothendieck category by \cite[prop.\ 3.12]{HJ-JLMS}.
  
  \item  $\Ch( {}_{ A }\!\Mod )^{ \opp }$ has small $\Hom$ sets because so does $\Ch( {}_{ A }\!\Mod )$.
  
  \item  $\LFunctorHomAngles{q}^{ \opp }$ respects set indexed limits because $\LFunctorHomAngles{q}$ respects set indexed colimits since it is exact and respects set indexed coproducts by Lemma \ref{lem:5i_5ii_30_32}(i).  \qedhere

\end{itemize}
\end{proof}

\begin{Remark}
\label{rmk:13_14_33_21A}
By symmetry, the functor $\LFunctorHomBrackets{q}: \Mod_{ Q,A } \xrightarrow{} \Ch( \Mod_A )$ has a left adjoint functor $\JFunctorHomBrackets{q}$ and a right adjoint functor $\MFunctorHomBrackets{q}$.
\end{Remark}

\begin{Remark}
Concrete formulae for $\JFunctorHomAngles{q}$ and $\MFunctorHomAngles{q}$ were provided in two specific cases in Sections \ref{subsec:differential_modules} and \ref{subsec:3-complexes}.  It may be interesting to provide formulae in higher generality, but we will not attempt to do so.
\end{Remark}

\begin{Remark}
There is a projective model category structure on ${}_{ Q,A }\!\Mod$; see \cite[thm.\ A]{HJ-JLMS}.  A special case of this is the standard model category structure on $\Ch( {}_{ A }\!\Mod )$; see \cite[sec.\ 2.3]{Hovey_Book}.  With respect to these abelian model category structures, $( \JFunctorHomAngles{q},\LFunctorHomAngles{q} )$ and $( \LFunctorHomAngles{q},\MFunctorHomAngles{q} )$ are Quillen adjunctions in the sense of \cite[def.\ 1.3.1]{Hovey_Book}, but we will not pursue this viewpoint.
\end{Remark}

\begin{Lemma}
\label{lem:19}
Let $q \in Q_0$ be given.  The functors $\JFunctorHomAngles{q}$ and $\MFunctorHomAngles{q}$ are exact.
\end{Lemma}

\begin{proof}
For each $p \in Q_0$ we have
\[
  \Hom_A\!\big( A,E_p\MFunctorHomAngles{q}( - ) \big)
  \cong \Hom_{ Q,A }\!\big( F_pA,\MFunctorHomAngles{q}( - ) \big)
  \cong \Hom_{ \Ch( {}_{ A }\!\Mod ) }( \LFunctorHomAngles{q}F_pA,- )
\]
by adjointness.  The right hand side is an exact functor because $\LFunctorHomAngles{q}F_pA$ is a projective object of $\Ch( {}_{ A }\!\Mod )$ by Lemma \ref{lem:5iii_b}(i) since $F_pA$ is in ${}_{ Q,A }\!\Prj$ by \cite[lem.\ 3.11]{HJ-JLMS}.  Hence the left hand side is an exact functor, and this implies that so is $E_p\MFunctorHomAngles{q}( - )$.  Since this holds for each $p \in Q_0$, the functor $\MFunctorHomAngles{q}$ is exact.

A similar argument shows that $\JFunctorHomAngles{q}$ is exact.
\end{proof}

\begin{Lemma}
\label{lem:18_39}
Given $q \in Q_0$, $j \in \BZ$, $X \in {}_{ Q,A }\!\Mod$, and $C \in \Ch( {}_{ A }\!\Mod )$, there are the following isomorphisms in ${}_{ \Bk }\!\Mod$:
\begin{align*}
  \Ext_{ \Ch( {}_{ A }\!\Mod ) }^j( C,\LFunctorHomAngles{q}X ) & \cong \Ext_{ Q,A }^j( \JFunctorHomAngles{q}C,X ), \\
  \Ext_{ \Ch( {}_{ A }\!\Mod ) }^j( \LFunctorHomAngles{q}X,C ) & \cong \Ext_{ Q,A }^j( X,\MFunctorHomAngles{q}C ).
\end{align*}
\end{Lemma}

\begin{proof}
The functor $\LFunctorHomAngles{q}$ is exact by Lemma \ref{lem:5i_5ii_30_32}(i), and the functors $\JFunctorHomAngles{q}$ and $\MFunctorHomAngles{q}$ are exact by Lemma \ref{lem:19}, so the isomorphisms hold by \cite[lem.\ 5.1]{HJ-Kyoto}.
\end{proof}

\begin{Lemma}
\label{lem:18A_36_55}
Let $q \in Q_0$ be given.  
\begin{enumerate}
\setlength\itemsep{4pt}

  \item  There is an isomorphism in ${}_{ Q,A }\!\Mod$,
\[
  \JFunctorHomAngles{q}( L \Tensor{A} C )
  \cong L \Tensor{A} \JFunctorHomAngles{q}( C ),
\]
natural with respect to $L \in {}_{ A }\!\Mod_A$ and $C \in \Ch( {}_{ A }\!\Mod )$.

  \item  There is an isomorphism in ${}_{ Q,A }\!\Mod$,
\[
  \MFunctorHomAngles{q}\Hom_A( L,C )
  \cong \Hom_A( L,\MFunctorHomAngles{q}C ),
\]
natural with respect to $L \in {}_{ A }\!\Mod_A$ and $C \in \Ch( {}_{ A }\!\Mod )$.

  \item  There is an isomorphism in $\Mod_{ Q,A }$,
\[
  \MFunctorHomBrackets{q}\Hom_A( C,L ) \cong \Hom_A\!\big( ( \JFunctorHomAngles{q}\Sigma C ) \circ \SerreFunctor^{ -1 },L \big),
\]
natural with respect to $L \in {}_{ A }\!\Mod_A$ and $C \in \Ch( {}_{ A }\!\Mod )$.

\end{enumerate}
\end{Lemma}

\begin{proof}
(i):  There are the following isomorphisms, natural with respect to $L \in {}_{ A }\!\Mod_A$, $C \in \Ch( {}_{ A }\!\Mod )$, and $X \in {}_{ Q,A }\!\Mod$:
\begin{align*}
  \Hom_{ Q,A }\!\big( \JFunctorHomAngles{q}( L \Tensor{A} C ),X \big)
  & \cong \Hom_{ \Ch( {}_{ A }\!\Mod ) }( L \Tensor{A} C,\LFunctorHomAngles{q}X ) && \mbox{Adjointness} \\
  & \cong \Hom_{ \Ch( {}_{ A }\!\Mod ) }\!\big( C,\Hom_A( L,\LFunctorHomAngles{q}X ) \big) && \mbox{Adjointness} \\
  & \cong \Hom_{ \Ch( {}_{ A }\!\Mod ) }\!\big( C,\LFunctorHomAngles{q}\Hom_A( L,X ) \big) && \mbox{Lemma \ref{lem:17A_29_54_57_38A_40A}(ii) and Remark \ref{rmk:17A_29_54_57_38A_40A}} \\
  & \cong \Hom_{ Q,A }\!\big( \JFunctorHomAngles{q}( C ),\Hom_A( L,X ) \big) && \mbox{Adjointness} \\
  & \cong \Hom_{ Q,A }\!\big( L \Tensor{A} \JFunctorHomAngles{q}( C ),X \big). && \mbox{Adjointness}
\end{align*}
This implies part (i) of the lemma.

(ii): There are the following isomorphisms, natural with respect to $L \in {}_{ A }\!\Mod_A$, $C \in \Ch( {}_{ A }\!\Mod )$, and $X \in {}_{ Q,A }\!\Mod$:
\begin{align*}
  \Hom_{ Q,A }\!\big( X,\MFunctorHomAngles{q}\Hom_A( L,C ) \big)
  & \cong \Hom_{ \Ch( {}_{ A }\!\Mod ) }\!\big( \LFunctorHomAngles{q}X,\Hom_A( L,C ) \big) && \mbox{Adjointness} \\
  & \cong \Hom_{ \Ch( {}_{ A }\!\Mod ) }( L \Tensor{A} \LFunctorHomAngles{q}X,C ) && \mbox{Adjointness} \\
  & \cong \Hom_{ \Ch( {}_{ A }\!\Mod ) }\!\big( \LFunctorHomAngles{q}( L \Tensor{A} X ),C \big) && \mbox{Lem.\ \ref{lem:17A_29_54_57_38A_40A}(i) and Rmk.\ \ref{rmk:17A_29_54_57_38A_40A}} \\
  & \cong \Hom_{ Q,A }( L \Tensor{A} X,\MFunctorHomAngles{q}C ) && \mbox{Adjointness} \\
  & \cong \Hom_{ Q,A }\!\big( X,\Hom_A( L,\MFunctorHomAngles{q}C ) \big). && \mbox{Adjointness}
\end{align*}
This implies part (ii) of the lemma.

(iii):  There are the following isomorphisms, natural with respect to $L \in {}_{ A }\!\Mod_A$, $C \in \Ch( {}_{ A }\!\Mod )$, and $Y \in \Mod_{ Q,A }$:
\begin{align*}
  & \Hom_{ Q^{ \opp },A^{ \opp } }\!\big( Y,\MFunctorHomBrackets{q}\Hom_A( C,L ) \big) && \\
  & \;\; \cong \Hom_{ \Ch( \Mod_A ) }\!\big( \LFunctorHomBrackets{q}Y,\Hom_A( C,L ) \big) && \mbox{Adjointness} \\
  & \;\; \cong \Hom_{ \Ch( {}_{ A }\!\Mod ) }\!\big( C,\Hom_{ A^{ \opp } }( \LFunctorHomBrackets{q}Y,L ) \big) && \mbox{\cite[prop.\ 4.4.10]{CFH-book}} \\
  & \;\; \cong \Hom_{ \Ch( {}_{ A }\!\Mod ) }\!\Big( C,\Sigma^{ -1 }\LFunctorHomAngles{q}\big( \Hom_{ A^{ \opp } }( Y,L ) \circ \SerreFunctor \big) \Big) && \mbox{Lemma \ref{lem:17A_29_54_57_38A_40A}(iv) and Remark \ref{rmk:17A_29_54_57_38A_40A}} \\
  & \;\; \cong \Hom_{ \Ch( {}_{ A }\!\Mod ) }\!\Big( \Sigma C,\LFunctorHomAngles{q}\big( \Hom_{ A^{ \opp } }( Y,L ) \circ \SerreFunctor \big) \Big) && \mbox{$\Sigma$ is an automorphism} \\
  & \;\; \cong \Hom_{ Q,A }\!\big( \JFunctorHomAngles{q}\Sigma C,\Hom_{ A^{ \opp } }( Y,L ) \circ \SerreFunctor \big) && \mbox{Adjointness} \\
  & \;\; \cong \Hom_{ Q,A }\!\big( ( \JFunctorHomAngles{q}\Sigma C ) \circ \SerreFunctor^{ -1 },\Hom_{ A^{ \opp } }( Y,L ) \big) && \mbox{$\SerreFunctor$ is an automorphism} \\
  & \;\; \cong \Hom_{ Q^{ \opp },A^{ \opp } }\!\Big( Y,\Hom_A\!\big( ( \JFunctorHomAngles{q}\Sigma C ) \circ \SerreFunctor^{ -1 },L \big) \Big). && \mbox{Lemma \ref{lem:16A_28_35_56}(iii) and Remark \ref{rmk:16A_28_35_56}}
\end{align*}
This implies part (iii) of the lemma.
\end{proof}

\begin{Lemma}
\label{lem:59_68}
Let $p,q \in Q_0$ be given and let $( \cX,\cY )$ be a cotorsion pair in ${}_{ A }\!\Mod$.
\begin{enumerate}
\setlength\itemsep{4pt}

  \item  If $C$ is a chain complex over $\cX\!$, then $( \JFunctorHomAngles{q}C )( p ) \in \cX\!$.

  \item  If $C$ is a chain complex over $\cY\!$, then $( \MFunctorHomAngles{q}C )( p ) \in \cY\!$.

\end{enumerate}
\end{Lemma}

\begin{proof}
(i): Let $C$ be a chain complex over ${}_{ A }\!\Mod$ and let $B$ be in ${}_{ A }\!\Mod$.  We can compute as follows by Lemmas \ref{lem:40} and \ref{lem:18_39}:
\begin{equation}
\label{equ:lem:59_68_lower:20}
  \Ext^1_A\!\big( E_p( \JFunctorHomAngles{q}C ),B \big)
  \cong \Ext^1_{ Q,A }( \JFunctorHomAngles{q}C,G_pB )
  \cong \Ext^1_{ \Ch( {}_{ A }\!\Mod ) }( C,\LFunctorHomAngles{q}G_pB ).
\end{equation}
Now suppose that $C$ is a chain complex over $\cX\!$.  Given $B \in \cY\!$, Lemma \ref{lem:5iii_a_70} combined with \cite[lem.\ 3.4]{HJ-TAMS} shows that $\LFunctorHomAngles{q}G_pB$ is a contractible chain complex over $\cY\!$.  Hence \cite[prop.\ 3.2]{Gillespie-degreewise} applied to the cotorsion pair $( \cX,\cY )$ gives that the right hand side of Equation \eqref{equ:lem:59_68_lower:20} is zero, whence the left hand side is also zero.  Since this holds for each $B \in \cY\!$, it follows that $E_p( \JFunctorHomAngles{q}C ) = ( \JFunctorHomAngles{q}C )( p )$ is in $\cX$ as claimed.

(ii): Proved similarly to part (i), but without the need to invoke \cite[lem.\ 3.4]{HJ-TAMS}.
\end{proof}

The next two lemmas have very similar proofs, and we provide only the first.

\begin{Lemma}
\label{lem:15A_lower}
Let $q \in Q_0$ and $C \in \Ch( {}_{ A }\!\Mod )$ be given.  Consider the following conditions:
\begin{enumerate}
\setlength\itemsep{4pt}

  \item  $C$ is acyclic.
  
  \item  For each $\ell \in \BZ$ we have $\JFunctorHomAngles{q}\Sigma^{ \ell }C \in {}_{ Q,A }\cE$.

\end{enumerate}
We have (i)$\Rightarrow$(ii), and if $P\langle q \rangle$ is good in the sense of Definition \ref{def:P_angles}, then we have (ii)$\Rightarrow$(i).
\end{Lemma}

\begin{Lemma}
\label{lem:20_49}
Let $q \in Q_0$ and $C \in \Ch( {}_{ A }\!\Mod )$ be given.  Consider the following conditions:
\begin{enumerate}
\setlength\itemsep{4pt}

  \item  $C$ is acyclic.
  
  \item  For each $\ell \in \BZ$ we have $\MFunctorHomAngles{q}\Sigma^{ \ell }C \in {}_{ Q,A }\cE$.

\end{enumerate}
We have (i)$\Rightarrow$(ii), and if $P\langle q \rangle$ is good in the sense of Definition \ref{def:P_angles}, then we have (ii)$\Rightarrow$(i).
\end{Lemma}

\begin{proof}[Proof of Lemma \ref{lem:15A_lower}]
We claim that the following is equivalent to condition (ii):
\begin{itemize}
\setlength\itemsep{4pt}

  \item[(ii')]  For each $\ell \in \BZ$, $p \in Q_0$, and $J \in {}_{ A }\!\Inj$ we have $\Ext_{ \Ch( {}_{ A }\!\Mod ) }^1( \Sigma^{ \ell }C,\LFunctorHomAngles{q}S_pJ ) = 0$.

\end{itemize}
The reason is that Lemma \ref{lem:18_39} gives
\begin{equation}
\label{equ:lem:15A_lower:10}
  \Ext_{ \Ch( {}_{ A }\!\Mod ) }^1( \Sigma^{ \ell }C,\LFunctorHomAngles{q}S_pJ )
  \cong \Ext_{ Q,A }^1( \JFunctorHomAngles{q}\Sigma^{ \ell }C,S_pJ ),
\end{equation}
so (ii') is equivalent to the condition that for each $\ell \in \BZ$, $p \in Q_0$, and $J \in {}_{ A }\!\Inj$ the right hand side of \eqref{equ:lem:15A_lower:10} is zero.  By Lemma \ref{lem:43} this is equivalent to (ii).

(i)$\Rightarrow$(ii'): Suppose that (i) holds.  Then (ii') is true because if $\ell \in \BZ$, $p \in Q_0$, and $J \in {}_{ A }\!\Inj$ are given, then the $\Ext$ group in (ii') is zero by \cite[prop.\ 2.3.4]{Garcia-Rozas-book} since $\Sigma^{ \ell }C$ is acyclic while $\LFunctorHomAngles{q}S_pJ$ is a semiinjective chain complex over ${}_{ A }\!\Mod$ by Lemma \ref{lem:17d}(ii).

Now suppose that $P\langle q \rangle$ is good.

(ii')$\Rightarrow$(i): Suppose that (ii') holds.  Let $\ell \in \BZ$ and $J \in {}_{ A }\!\Inj$ be given.  Since $P\langle q \rangle$ is good, Lemma \ref{lem:9A_13A} says that the chain complex $\LFunctorHomAngles{q}S_qJ$ has $J$, viewed as a chain complex concentrated in degree zero, as a direct summand, so (ii') implies
\[
  \Ext_{ \Ch( {}_{ A }\!\Mod ) }^1( \Sigma^{ \ell }C,J ) = 0,
\]
which again implies
\[
  \Ext_{ \Ch( {}_{ A }\!\Mod ) }^1( C,\Sigma^{ -\ell }J ) = 0
\]
because $\Sigma$ is an automorphism of the abelian category $\Ch( {}_{ A }\!\Mod )$.  The last equation holds for each $\ell \in \BZ$, so Lemma \ref{lem:43} applied to the case of chain complexes shows that (i) holds.
\end{proof}

\section{Proof of Theorems \ref{thm:A} and \ref{thm:B}}
\label{sec:proof_of_Thm_B}

This section proves Theorems \ref{thm:A} and \ref{thm:B} by way of Theorems \ref{thm:31_50}, \ref{thm:58_60}, and \ref{thm:81_19A}, whose proofs are closely related but sufficiently short and different that we provide all three.

\begin{Remark}
\label{rmk:functors_from_a_symmetric_bimodule}
If $A$ is a commutative ring and $M$ is in ${}_{ A }\!\Mod$ or $\Mod_A$, then we will view $M$ as a symmetric bimodule in ${}_{ A }\!\Mod_A$.  This provides us with functors 
\begin{gather*}
  \Hom_A( -,M ) : {}_{ Q,A }\!\Mod \xrightarrow{} \Mod_{ Q,A }, \\[1mm]
  \Hom_A( M,- ) : {}_{ Q,A }\!\Mod \xrightarrow{} {}_{ Q,A }\!\Mod, \\[1mm]
  M \Tensor{A} - : {}_{ Q,A }\!\Mod \xrightarrow{} {}_{ Q,A }\!\Mod.
\end{gather*}
\end{Remark}

\begin{Theorem}
\label{thm:31_50}
Assume that $A$ is a commutative noetherian ring and consider the following conditions:
\begin{enumerate}
\setlength\itemsep{4pt}

  \item  $A$ is a Gorenstein ring.
  
  \item  If $L \in {}_{ A }\!\Prj$ and $P \in {}_{ Q,A }\!\Mod$ are given where $P( p ) \in {}_{ A }\!\Prj$ for each $p \in Q_0$, then
\[  
  P \in {}_{ Q,A }\cE \Rightarrow \Hom_A( P,L ) \in \cE_{ Q,A }
\]
(this makes sense because $\Hom_A( P,L )$ is in $\Mod_{ Q,A }$ by Remark \ref{rmk:functors_from_a_symmetric_bimodule}).

\end{enumerate}
We have (i)$\Rightarrow$(ii).  

If there exists $q \in Q_0$ such that $P\langle q \rangle$ is good in the sense of Definition \ref{def:P_angles}, then we have (ii)$\Rightarrow$(i).  This applies in particular if there exists $q \in Q_0$ such that $\SerreFunctor q \neq q$; see Remark \ref{rmk:P}.
\end{Theorem}

\begin{proof}
(i)$\Rightarrow$(ii): Suppose that (i) holds.  Let $L \in {}_{ A }\!\Prj$ and $P \in {}_{ Q,A }\!\Mod$ be given where $P( p ) \in {}_{ A }\!\Prj$ for each $p \in Q_0$.  We can argue as follows:
\begin{align*}
  P \in {}_{ Q,A }\cE
  & \Leftrightarrow \forall q \in Q_0: \LFunctorHomAngles{q}P \mbox{ is acyclic} && \mbox{Lemma \ref{lem:5v}} \\
  & \stackrel{ \rm (a) }{ \Rightarrow } \forall q \in Q_0: \Hom_A( \LFunctorHomAngles{q}P,L ) \mbox{ is acyclic} && \mbox{\cite[thm.\ 19.5.25]{CFH-book}} \\
  & \Leftrightarrow \forall q \in Q_0: \Sigma^{ -1 }\LFunctorHomBrackets{q}\big( \Hom_A( P,L ) \circ \SerreFunctor^{ -1 } \big) \mbox{ is acyclic} && \mbox{Lemma \ref{lem:17A_29_54_57_38A_40A}(iii) and Remark \ref{rmk:17A_29_54_57_38A_40A}} \\
  & \Leftrightarrow \forall q \in Q_0: \LFunctorHomBrackets{q}\big( \Hom_A( P,L ) \circ \SerreFunctor^{ -1 } \big) \mbox{ is acyclic} && \\
  & \Leftrightarrow \Hom_A( P,L ) \circ \SerreFunctor^{ -1 } \in \cE_{ Q,A } && \mbox{Lemma \ref{lem:5v} applied to $Q^{ \opp }$} \\
  & \Leftrightarrow \Hom_A( P,L ) \in \cE_{ Q,A }. && \mbox{$\SerreFunctor$ is an automorphism}
\end{align*}
To get the implication marked (a), we appeal to \cite[thm.\ 19.5.25]{CFH-book}.  This requires $\LFunctorHomAngles{q}P$ to be a chain complex over ${}_{ A }\!\Prj$, which it is by Lemma \ref{lem:5i_5ii_30_32}(ii).

Now assume that $q \in Q_0$ is such that $P\langle q \rangle$ is good.

(ii)$\Rightarrow$(i): Suppose that (ii) holds.  Let $L \in {}_{ A }\!\Prj$ and a chain complex $C$ over ${}_{ A }\!\Prj$ be given.  By \cite[thm.\ 19.5.25]{CFH-book} it is enough to show that if $C$ is acyclic, then $\Hom_A( C,L )$ is acyclic, and we can argue as follows: 
\begin{align*}
  C \mbox{ is acyclic}
  & \Leftrightarrow \forall \ell \in \BZ: \JFunctorHomAngles{q}\Sigma^{ \ell }C \in {}_{ Q,A }\cE && \mbox{Lemma \ref{lem:15A_lower}} \\
  & \Leftrightarrow \forall \ell \in \BZ: ( \JFunctorHomAngles{q}\Sigma^{ \ell }C ) \circ \SerreFunctor^{ -1 } \in {}_{ Q,A }\cE && \mbox{$\SerreFunctor$ is an automorphism} \\
  & \stackrel{ \rm (b) }{ \Rightarrow } \forall \ell \in \BZ: \Hom_A\!\big( ( \JFunctorHomAngles{q}\Sigma^{ \ell }C ) \circ \SerreFunctor^{ -1 } ,L \big) \in \cE_{ Q,A } && \mbox{Condition (ii)} \\
  & \Leftrightarrow \forall \ell \in \BZ: \MFunctorHomBrackets{q}\Hom_A( \Sigma^{ \ell-1 }C,L ) \in \cE_{ Q,A } && \mbox{Lemma \ref{lem:18A_36_55}(iii)} \\
  & \Leftrightarrow \Hom_A( C,L ) \mbox{ is acyclic.} && \mbox{Lemma \ref{lem:20_49} applied to $Q^{ \opp }$}
\end{align*}
To get the implication marked (b), we appeal to condition (ii).  This requires that $\big( ( \JFunctorHomAngles{q}\Sigma^{ \ell }C ) \circ \SerreFunctor^{ -1 } \big)( p ) \in {}_{ A }\!\Prj$ for each $p \in Q_0$.  This is equivalent to $( \JFunctorHomAngles{q}\Sigma^{ \ell }C )( p ) \in {}_{ A }\!\Prj$ for each $p \in Q_0$, which is true by Lemma \ref{lem:59_68}(i) applied to the projective cotorsion pair $( {}_{ A }\!\Prj,{}_{ A }\!\Mod )$.
\end{proof}

\begin{Theorem}
\label{thm:58_60}
Assume that $A$ is a commutative noetherian ring and consider the following conditions:
\begin{enumerate}
\setlength\itemsep{4pt}

  \item  $A$ is a Gorenstein ring.
  
  \item  If $J \in {}_{ A }\!\Inj$ and $I \in {}_{ Q,A }\!\Mod$ are given where $I( p ) \in {}_{ A }\!\Inj$ for each $p \in Q_0$, then
\[  
  I \in {}_{ Q,A }\cE \Rightarrow \Hom_A( J,I ) \in {}_{ Q,A }\cE
\]
(this makes sense because $\Hom_A( J,I )$ is in ${}_{ Q,A }\!\Mod$ by Remark \ref{rmk:functors_from_a_symmetric_bimodule}).
  
\end{enumerate}
We have (i)$\Rightarrow$(ii).

If there exists $q \in Q_0$ such that $P\langle q \rangle$ is good in the sense of Definition \ref{def:P_angles}, then we have (ii)$\Rightarrow$(i).  This applies in particular if there exists $q \in Q_0$ such that $\SerreFunctor q \neq q$; see Remark \ref{rmk:P}.
\end{Theorem}

\begin{proof}
(i)$\Rightarrow$(ii): Suppose that (i) holds.  Let $J \in {}_{ A }\!\Inj$ and $I \in {}_{ Q,A }\!\Mod$ be given where $I( p ) \in {}_{ A }\!\Inj$ for each $p \in Q_0$.  We can argue as follows:
\begin{align*}
  I \in {}_{ Q,A }\cE
  & \Leftrightarrow \forall q \in Q_0: \LFunctorHomAngles{q}I \mbox{ is acyclic} && \mbox{Lemma \ref{lem:5v}} \\
  & \stackrel{ \rm (a) }{ \Rightarrow } \forall q \in Q_0: \Hom_A( J,\LFunctorHomAngles{q}I ) \mbox{ is acyclic} && \mbox{\cite[thm.\ 19.5.25]{CFH-book}} \\
  & \Leftrightarrow \forall q \in Q_0: \LFunctorHomAngles{q}\Hom_A( J,I ) \mbox{ is acyclic} && \mbox{Lemma \ref{lem:17A_29_54_57_38A_40A}(ii) and Remark \ref{rmk:17A_29_54_57_38A_40A}} \\
  & \Leftrightarrow \Hom_A( J,I ) \in {}_{ Q,A }\cE. && \mbox{Lemma \ref{lem:5v}}
\end{align*}
To get the implication marked (a), we appeal to \cite[thm.\ 19.5.25]{CFH-book}.  This requires $\LFunctorHomAngles{q}I$ to be a chain complex over ${}_{ A }\!\Inj$, which it is by Lemma \ref{lem:5i_5ii_30_32}(ii).

Now assume that $q \in Q_0$ is such that $P\langle q \rangle$ is good.

(ii)$\Rightarrow$(i): Suppose that (ii) holds.  Let $J \in {}_{ A }\!\Inj$ and a chain complex $C$ over ${}_{ A }\!\Inj$ be given.  By \cite[thm.\ 19.5.25]{CFH-book} it is enough to show that if $C$ is acyclic, then $\Hom_A( J,C )$ is acyclic, and we can argue as follows: 
\begin{align*}
  C \mbox{ is acyclic}
  & \Leftrightarrow \forall \ell \in \BZ: \MFunctorHomAngles{q}\Sigma^{ \ell }C \in {}_{ Q,A }\cE && \mbox{Lemma \ref{lem:20_49}} \\
  & \stackrel{ \rm (b) }{ \Rightarrow } \forall \ell \in \BZ: \Hom_A( J,\MFunctorHomAngles{q}\Sigma^{ \ell }C ) \in {}_{ Q,A }\cE && \mbox{Condition (ii)} \\
  & \Leftrightarrow \forall \ell \in \BZ: \MFunctorHomAngles{q}\Hom_A( J,\Sigma^{ \ell }C ) \in {}_{ Q,A }\cE && \mbox{Lemma \ref{lem:18A_36_55}(ii)} \\
  & \Leftrightarrow \Hom_A( J,C ) \mbox{ is acyclic.} && \mbox{Lemma \ref{lem:20_49}}
\end{align*}
To get the implication marked (b), we appeal to condition (ii).  This requires $( \MFunctorHomAngles{q}\Sigma^{ \ell }C )( p ) \in {}_{ A }\!\Inj$ for each $p \in Q_0$, which is true by Lemma \ref{lem:59_68}(ii) applied to the injective cotorsion pair $( {}_{ A }\!\Mod,{}_{ A }\!\Inj )$.
\end{proof}

In the following theorem, ${}_{ A }\!\Flat$ denotes the category of flat left $A$-modules.

\begin{Theorem}
\label{thm:81_19A}
Assume that $A$ is a commutative noetherian ring and consider the following conditions:
\begin{enumerate}
\setlength\itemsep{4pt}

  \item  $A$ is a Gorenstein ring.
  
  \item  If $J \in \Inj_A$ and $F \in {}_{ Q,A }\!\Mod$ are given where $F( p ) \in {}_{ A }\!\Flat$ for each $p \in Q_0$, then
\[  
  F \in {}_{ Q,A }\cE \Rightarrow J \Tensor{A} F \in {}_{ Q,A }\cE
\]
(this makes sense because $J \Tensor{A} F$ is in ${}_{ Q,A }\!\Mod$ by Remark \ref{rmk:functors_from_a_symmetric_bimodule}).

\end{enumerate}
We have (i)$\Rightarrow$(ii).  

If there exists $q \in Q_0$ such that $P\langle q \rangle$ is good in the sense of Definition \ref{def:P_angles}, then we have (ii)$\Rightarrow$(i).  This applies in particular if there exists $q \in Q_0$ such that $\SerreFunctor q \neq q$; see Remark \ref{rmk:P}.
\end{Theorem}

\begin{proof}
(i)$\Rightarrow$(ii): Suppose that (i) holds.  Let $J \in \Inj_A$ and $F \in {}_{ Q,A }\!\Mod$ be given where $F( p ) \in {}_{ A }\!\Flat$ for each $p \in Q_0$.  We can argue as follows:
\begin{align*}
  F \in {}_{ Q,A }\cE
  & \Leftrightarrow \forall q \in Q_0: \LFunctorHomAngles{q}F \mbox{ is acyclic} && \mbox{Lemma \ref{lem:5v}} \\
  & \stackrel{ \rm (a) }{ \Rightarrow } \forall q \in Q_0: J \Tensor{A} \LFunctorHomAngles{q}F \mbox{ is acyclic} && \mbox{\cite[thm.\ 19.5.25]{CFH-book}} \\
  & \Leftrightarrow \forall q \in Q_0: \LFunctorHomAngles{q}( J \Tensor{A} F ) \mbox{ is acyclic} && \mbox{Lemma \ref{lem:17A_29_54_57_38A_40A}(i) and Remark \ref{rmk:17A_29_54_57_38A_40A}} \\
  & \Leftrightarrow J \Tensor{A} F \in {}_{ Q,A }\cE. && \mbox{Lemma \ref{lem:5v}}
\end{align*}
To get the implication marked (a), we appeal to \cite[thm.\ 19.5.25]{CFH-book}.  This requires $\LFunctorHomAngles{q}F$ to be a chain complex over ${}_{ A }\!\Flat$, which it is by Lemma \ref{lem:5i_5ii_30_32}(ii).

Now assume that $q \in Q_0$ is such that $P\langle q \rangle$ is good.

(ii)$\Rightarrow$(i): Suppose that (ii) holds.  Let $J \in \Inj_A$ and a chain complex $C$ over ${}_{ A }\!\Flat$ be given.  By \cite[thm.\ 19.5.25]{CFH-book} it is enough to show that if $C$ is acyclic, then $J \Tensor{A} C$ is acyclic, and we can argue as follows: 
\begin{align*}
  C \mbox{ is acyclic}
  & \Leftrightarrow \forall \ell \in \BZ: \JFunctorHomAngles{q}( \Sigma^{ \ell }C ) \in {}_{ Q,A }\cE && \mbox{Lemma \ref{lem:15A_lower}} \\
  & \stackrel{ \rm (b) }{ \Rightarrow } \forall \ell \in \BZ: J \Tensor{A} \JFunctorHomAngles{q}( \Sigma^{ \ell }C ) \in {}_{ Q,A }\cE && \mbox{Condition (ii)} \\
  & \Leftrightarrow \forall \ell \in \BZ: \JFunctorHomAngles{q}( J \Tensor{A} \Sigma^{ \ell }C ) \in {}_{ Q,A }\cE && \mbox{Lemma \ref{lem:18A_36_55}(i)} \\
  & \Leftrightarrow J \Tensor{A} C \mbox{ is acyclic.} && \mbox{Lemma \ref{lem:15A_lower}}
\end{align*}
To get the implication marked (b), we appeal to condition (ii).  This requires $\big( \JFunctorHomAngles{q}( \Sigma^{ \ell }C ) \big)( p ) \in {}_{ A }\!\Flat$ for each $p \in Q_0$, which is true by Lemma \ref{lem:59_68}(i) applied to the so-called flat cotorsion pair $( {}_{ A }\!\Flat,{}_{ A }\!\Cot )$; see \cite[lem.\ 3.4.1]{Xu-Flat-covers}.
\end{proof}

\begin{proof}[Proof of Theorem \ref{thm:B}]
Combine Theorems \ref{thm:31_50}, \ref{thm:58_60}, and \ref{thm:81_19A}, noting that, for instance, a $Q$-shaped diagram of projective $A$-modules is precisely an object $P \in {}_{ Q,A }\!\Mod$ where $P( p ) \in {}_{ A }\!\Prj$ for each $p \in Q_0$.
\end{proof}

\begin{proof}[Proof of Theorem \ref{thm:A}]
Let $\Bk$ be $\BZ$ and let $Q$ be $\QNcpx$, the $\BZ$-linear category given by the quiver in Figure \ref{fig:linear_quiver} with the relations that $N$ consecutive arrows compose to zero.  Then $Q$-shaped diagrams are $N$-complexes.  Theorem \ref{thm:B} implies Theorem \ref{thm:A} because $\SerreFunctor q \neq q$ is satisfied for each $q \in Q_0$, and because ${}_{ Q,A }\cE$ consists precisely of the $N$-acyclic $N$-complexes; see \cite[2.5 and 3.5]{HJ-Abel}.
\end{proof}

\medskip
\noindent
{\bf Acknowledgement.}
The first author was supported by the Danish National Research Foundation through the Copenhagen Centre for Geometry and Topology (DNRF151).  The second author was supported by a Research Project 2 from the Independent Research Fund Denmark (grant 1026-00050B) and by Aarhus University Research Foundation (grant AUFF-E-2024-9-43).

\end{document}